\documentclass[11pt]{article}
\usepackage{amsmath,amssymb,mathrsfs,hyperref,color,amsthm}
\usepackage{mathtools}
\usepackage[x11names]{xcolor}
\usepackage{paralist}

\usepackage{subfigure}
\usepackage{float}
\usepackage{tikz}
\usepackage{cases}
\usetikzlibrary{intersections}
\usepackage{xcolor}
\usepackage[latin1]{inputenc}
\usepackage{bbm}
\usepackage{dsfont}
\usepackage[T1]{fontenc}

\usetikzlibrary{shadings}
\usetikzlibrary[intersections]
\usetikzlibrary{arrows,shapes}

\makeatletter
\def\author#1{\gdef\autrun{\def\and{\unskip, }#1}\gdef\@author{#1}}

\makeatother
\def\email#1{e-mail: #1}

\def\keywords#1{\par\medskip
\noindent\textbf{Keywords.} #1}

\newtheorem{lemma}{Lemma}[section]
\newtheorem{theorem}{Theorem}
\newtheorem{definition}{Definition}[section]
\newtheorem{proposition}{Proposition}[section]
\newtheorem{remarks}{Remark}[section]
\newtheorem{corollary}{Corollary}[section]

\newcommand{\Lip}{\mathrm{Lip}}
\newcommand{\T}{\mathbb{T}}
\newcommand{\R}{\mathbb{R}}
\newcommand{\Z}{\mathbb{Z}}
\newcommand{\N}{\mathbb{N}}

\newcommand{\PP}{\mathcal{P}}
\newcommand{\MMM}{\mathcal{M}}

\newcommand{\CC}{\mathcal{C}}
 
\newcommand{\M}{\mathcal{M}}

\newcommand{\LL}{\mathcal{L}}

\DeclareMathOperator*{\argmin}{argmin} 
\DeclareMathOperator*{\supp}{spt}

\numberwithin{equation}{section}

\begin{document}

%%%%% To ease editing, add:

\baselineskip=15pt

%%%%%%%%%%%%%%%%

%% In the running head, give an abbreviation of the title. 
%\titlerunning{Weak KAM solutions of contact Hamilton-Jacobi equations}

\title{Relaxed Lagrangian Approach to First-Order Non-Convex Mean Field Type Control Problem}

\author{Cristian Mendico \thanks{Institut de Math\'ematique de Bourgogne, UMR 5584 CNRS, Universit\'e Bourgogne, 21000 Dijon, France; \email{\tt cristian.mendico@u-bourgogne.fr}} \and Kaizhi Wang \thanks{School of Mathematical Sciences, CMA-Shanghai, Shanghai Jiao Tong University, Shanghai 200240, China; \email{\tt kzwang@sjtu.edu.cn}} \and  Yuchen Xu \thanks{School of Mathematical Sciences, CMA-Shanghai, Shanghai Jiao Tong University, Shanghai 200240, China; \email{\tt math\_rain@sjtu.edu.cn}}}

%\date{\today}

\maketitle

%\address{Cristian Mendico: Dipartimento di Matematica, Universit\`a di Roma ``Tor Vergata'', Via della Ricerca Scientifica 1, 00133 Roma, Italy; \email{mendico@mat.uniroma2.it}
%\and Kaizhi Wang: School of Mathematical Sciences, Shanghai Jiao Tong University, Shanghai 200240, China; \email{kzwang@sjtu.edu.cn}}
%\and Jun Yan: School of Mathematical Sciences, Shanghai Jiao Tong University, Shanghai 200240, China; \email{math_rain@sjtu.edu.cn}}

%\subjclass{}

%%%%%%%%
\begin{abstract}
This paper addresses the existence of equilibria for Mean Field type Control problems of first-order with non-convex action functional. Introducing a relaxed Lagrangian approach on the Wasserstein space to handle the lack of convexity, we prove the existence of new relaxed Nash equilibria and we show that our existence result encompasses the classical Mean Field Control problem's existence result under convex data conditions.


\end{abstract}

%\tableofcontents

%%%%%%%%%%%%%%%%%%%%%%%%%%%%%%%%%%%%%
%Section 1
\section{Introduction} \label{introduction}
\setcounter{equation}{0}
%The theory of Mean Field Games (MFG) deals with the analysis of multi-agent dynamical systems involving an infinite number of rational, indistinguishable and uncooperative players. This theory was proposed by Lasry and Lions \cite{MR2269875, MR2271747, MR2295621} and Huang, Malham\'e and Caines \cite{MR2346927} independently. The classical MFG Problem is formulated by a backward Hamilton-Jacobi equation and a forward Fokker-Planck equation modeling the time evolution of the players' optimal strategies, on the one hand, and the evolution of players in space according to such strategies, that is,
%\begin{equation}\label{MFG}
%\begin{cases}
%-\partial_{t} u(t, x) + H(x, Du(t,x), m(t))= 0, & (t, x) \in [0, T] \times \Omega
%\\
%\partial_{t} m(t) - \textrm{div}(m(t) D_{p}H(x, Du(t, x), m(t))) = 0, & (t, x) \in [0, T] \times \Omega
%\\
%u(T, x) = u_T(x), \quad \textrm{for}\; x \in \Omega, \quad m(0) = m_0. 
%\end{cases}
%\end{equation}
%In particular, following the classical notion of Nash equilibria, we say that the game \eqref{MFG} reaches a mean field equilibrium when all strategies and the mean field distribution are consistent, and no agent can unilaterally improve their outcome.
The motivation for this work stems from recent advances in deep learning (DL) and the application of techniques from optimal control and mean field theory. Specifically, deep neural networks can be considered as discrete-time nonlinear dynamical systems and their internal parameters can be recast as controls, allowing the formulation of the training process as an optimal control problem. Given a set of training data $(\xi, \zeta) \in \Omega$ distributed according to $\mathcal{M} \in \PP(\Omega)$ and an activation function $\varphi$, we construct a neural network with input $z \in \R^d$ and output
\begin{equation}\label{RESNET1}
	\frac{1}{n} \sum_{i=1}^{n} \beta_{t, i}\varphi(\alpha_{t,  i} z + \rho_{t,  i}\zeta_{t}) = \int_{\R^d} \beta \varphi(\alpha z + \rho\zeta_t)\; \nu^n_t(d\alpha, d\beta, d\rho),
\end{equation}
where $t$ denotes the index of the layer, $(\alpha, \beta, \rho)$ are the data of the neural network and 
\begin{equation*}
	%\nu_{t, n}(d\alpha, d\beta, d\rho) = \frac{1}{n} \sum_{i=1}^{n} \delta_{\{\alpha_{t, i}, \beta_{t, i}, \rho_{t, i}\}}. 
	\nu^n_t = \frac{1}{n} \sum_{i=1}^{n} \delta_{\{\alpha_{t,  i}, \beta_{t,  i}, \rho_{t,  i}\}}.
\end{equation*}
Thus, setting $a= (\alpha, \beta, \rho)$ and $\psi(z, a, \zeta_t) = \beta\varphi(\alpha z + \rho\zeta_t)$, the architecture of the neural network can be written as 
\begin{equation*}
	X_{t_{k+1}}^{\nu^n, \xi, \zeta} = X_{t_{k}}^{\nu^n, \xi, \zeta} + \frac{\Delta t}{n} \sum_{i=1}^{n} \psi(X_{t_{k}}^{\nu^n, \xi, \zeta}, a_{t_k,i}, \zeta_{t_{k}}),
\end{equation*}
where $\Delta t = t_{k+1} - t_k$. Therefore, as $\Delta t  \downarrow 0$ and $n \uparrow \infty$ (heuristically) we obtain 
\begin{equation*}
	X_{t}^{\nu, \xi, \zeta} = \xi + \int_{0}^{t}\int \psi(X_{r}^{\nu, \xi, \zeta}, a, \zeta_r)\; \nu_r(da)dr
\end{equation*}
and the training procedure of such a model can be recast as an optimization problem on the Wasserstein space
\begin{equation}\label{RESNET2}
	\min_{\nu} \int_{\Omega} g(\xi, \zeta, \nu, X_{T}^{\nu, \xi, \zeta})\; \mathcal{M}({d\xi, d\zeta}). 
\end{equation}
So, looking at the minimization problem \eqref{RESNET2}, in this work we aim to extend the possibility of studying neural network architectures by means of the relaxed Lagrangian approach to the case of mean field state equation modeling, for instance, congestion of data or sparsity constraint.

%Mean Field Control (MFC), as the control-theoretic counterpart of MFG, focuses on optimizing the behavior of a centralized player, who influences a continuum of agents through a common policy.
%The equilibria of \eqref{MFG} is equivalent to looking at the minimizers of the following MFC problem
Having these models in mind, Mean Field Control (MFC) is a recent branch of control theory that deals with the optimal control of a large number of interacting agents. It arises when considering systems with a population of individual decision-makers (agents), where each agent's dynamics and objectives depend not only on their own states and actions but also on the aggregate behavior of the entire population.
From a variational viewpoint, a general MFC problem of first order (having deterministic dynamics) can be written as a minimization problem of the following type
\begin{equation*}
	\inf_{\dot{\gamma}} \int_0^T L(\gamma(t), \dot\gamma(t), m(t))\; dt + u_T(\gamma(T)),
\end{equation*}
where $m \in C([0, T], \PP(\Omega))$ is the probability distribution of all agents and $L$ is the Legendre transform of $H$. Here, the Lagrangian $L$ is convex in the control variable $\dot{\gamma}$.

We consider the MFC problem with general nonlinear controlled dynamics, whose state equations directly depend on the distribution of players, of the form
\begin{equation}\label{extantion}
	\min_{u: [0,T] \rightarrow \R^n} \int_0^T L(\gamma(t), u(t), m(t))\; dt  \quad\textrm{subject to} \quad \dot\gamma (t) = f(\gamma(t), u(t), m(t)),
\end{equation}
where the Lagrangian $L: \T^d \times \R^n \times \PP(\T^d) \to \R$ is non-convex with respect to the control variable $u(t) \in \R^n$.
The convexity assumption of the Lagrangian function, with respect to the control, is needed even if this is not always the case when considering some real life applications such as biological modeling, pricing dynamics, multi-agent reinforcement learning and deep learning (\cite{zbMATH06721976,MR3627592,MR3891852,MR4215224,lee2020controlling,martinez2024symmetries,MR4550416}).
For the case that $L$ is convex, such a problem is classically solved by using the so-called Lagrangian approach, see for instance \cite{zbMATH07061009, MR4132067, zbMATH07408117, zbMATH07139232}.
In case of stochastic systems, such an analysis has been developed in \cite{zbMATH06195704, 
	zbMATH06227443, zbMATH06721976, zbMATH06721977, zbMATH06191712}. However, to the best of the authors' knowledge, no direct results are available for the first-order case.

We address this challenge by generalizing the Lagrangian approach to non-convex mean field optimization problems taking inspiration from the well-known relaxed control approach \cite{zbMATH04030334, zbMATH00042424, zbMATH01341810}.
The relaxation method, we propose here, transforms the non-convex optimal control problem (\ref{extantion}) into an optimal control problem over probability measures on the control space (\ref{optimization problem}). %Such an approach was first proposed by Laker \cite{THM3.7}, where the author dealt with the existence of MFC equilibria for McKean-Vlasov type problems. Then, in \cite{MR3891852} the authors used the relaxed Lagrangian approach to derive the Hamilton-Jacobi-Bellman equation and the Pontryagin optimality conditions for resNet type MFC problems. Following this,  in \cite{jabir2021meanfieldneuralodesrelaxed} the problem of finding Pontryagin optimality conditions for bayesian neural type MFC problems was addressed. Notably, both the Lagrangian and the state equation in the above manuscripts are not directly related with the distribution of all players.
In particular, once the Nash equilibria is found we have, consequently, constructed the optimal neural network architecture for the data we are interested in.

\subsection*{Relaxation method}
Here, we proceed to show the relaxation procedure on Wasserstein space. 

Fix a time horizon $T>0$. Let $m_0 \in \PP \left( \T^d \right)$ denote the initial distribution of all agents. Denote by $\MMM( \left[ 0, T \right] \times \R^n)$ the space of all Borel measures on $\left[ 0, T \right] \times \R^n$. 
Define the set $\PP_U$ by
$$
\PP_U : = \left\{ \mu \in \MMM( \left[ 0, T \right] \times \R^n)\, \middle \vert \,\, \exists \left(\mu_t\right)_{t \in [0,T]} \subset \PP \left( \R^n \right), s.t.\,\, \mu = \LL_{[0,T]} \otimes \mu_t \right\},
$$
where $\mu = \LL_{[0,T]} \otimes \mu_t$ means that for any Borel map $g: [0,T] \times \R^n \rightarrow [0,+\infty]$, the following holds
$$
\int_{[0,T] \times \R^n} g(t,u) \mu(dt,du) = \int_0^T \int_{\R^n} g(t,u) \mu_t(du)dt.
$$
Similar to (\ref{extantion}), we consider the following relaxed state equation:
\begin{equation}\label{traj_1}
	\begin{cases}
		\dot\gamma(t) = \int_{\R^n} f \left( \gamma(t), u,m(t)\right) \mu_t(du) , & t \in [0, T]
		\\
		\gamma(0) = x, & x \in \T^d
	\end{cases}
\end{equation}
where $\mu \in \PP_U$ represents the strategy and $m \in \CC ( [0,T], \PP( \T^d))$ is the distribution of all agents with $m(0)=m_0$.
Therefore, the aim of each agent is to choose a strategy such that the cost function attains the minimum, which is defined by
%Each rational, indistinguishable and uncooperative player $i$ has private state process $\gamma^i$, whose dynamic system is given by (\ref{traj_1}). % with the initial state $x \in \T^d$, the strategy $\mu^i \in \PP_U$ and the evolution of the distribution $m \in \CC \left([0,T] , \PP \left( \T^d \right) \right)$. 
%Player $i$ aims to find an optimal strategy $\mu^{i*} \in \PP_U$ such that the cost function attains the minimum, which is defined by
$$
J^m (\gamma, \mu) = \int_0^T \int_{\R^n} L \left( \gamma (t), u, m(t) \right) \mu (dt,du).
$$
%attains the minimum. where $\gamma^i$ and $m$ are both related to $\mu^i$.

Define
\begin{align*}
	\Gamma_T : = \left\{ \gamma \in  AC \left( [0,T], \T^d \right)\, \right\vert \,\,
	& \exists (x,\mu) \in \T^d \times \PP_U, \exists m \in \CC \left( [0,T], \PP\left( \T^d \right) \right)\,\, \text{with} \,\, m(0) = m_0, \\
	&\left. s.t. \, \gamma \,\, \text{is the solution of the state equation} \, (\ref{traj_1}) \right\}.		
\end{align*}
Let $\Gamma_T$ be endowed with the uniform norm $\Vert \cdot \Vert _{\infty}$.
Let $m\in \CC \left([0,T], \PP \left( \T^d \right) \right)$ be the map such that $m(0)=m_0$, and let $P \in \PP \left( \Gamma_T \times \PP_U \right)$ be the probability measure such that every state-measure pair $\left( \gamma, \mu \right) \in \supp(P)$ satisfies
\begin{equation} \label{222_1}
	\dot{\gamma}(t) = \int_{\R^n} f \left( \gamma(t), u, m(t) \right) \mu_t(du).%, \quad \forall t \in [0,T].
\end{equation}
The total cost of all agents is defined by
$$
J\left( m,P \right) : = \int_{\Gamma_T \times \PP_U} J^m \left( \gamma, \mu \right) P \left( d \gamma, d \mu \right).
$$
Since the map $m$ represents the state distribution of all agents while the support of $P$ is the state processes and the strategies of all agents, % there exist some relations between these two variables. Therefore, 
the MFC problem is
\begin{equation} \label{optimization problem}
\inf_{P} J\left( \left( e_t \sharp \pi_1 \sharp P \right)_{t \in [0,T]},P \right),
\end{equation}
where the infimum is taken over all the measures $P \in \PP \left( \Gamma_T \times \PP_U \right)$ such that any state-measure pair $\left( \gamma, \mu \right) \in \supp(P)$ satisfies (\ref{222_1}). See Appendix \ref{assumtions and definitions} for details of the evaluation map $e_t$ and the push-forward measure $\pi_1 \sharp P \in \PP (\Gamma_T)$.
%The evaluation map $e_t$ is defined by
%$$
%e_t(\gamma) = \gamma(t), \quad \forall \gamma \in \Gamma_T, \forall t \in [0,T],
%$$
%and the push-forward measure $\pi_1 \sharp P \in \PP (\Gamma_T)$ is defined by
%$$
%\pi_1 \sharp P \left( B \right) : = P \left( \pi_1^{-1} \left(B \right) \right), \quad \forall B \in \mathcal{B}\left( \Gamma_T \right).
%$$

\subsection{Assumptions}
\label{assumptions}
%the assumptions for f and L
%{\it From now on, fix three constants $p,q, p^\prime \in \R$ with $q>p \geq 1$ and $q \geq p^\prime \geq 1$.}
\textbf{From now on, fix two constants $p,q \in \R$ with $q>p \geq 1$.}

\medskip

Let $L(x,u,\nu): \T^d \times \R^n \times \PP \left( \T^d \right) \to \R$ be a Lagrangian satisfying the following assumptions.
\begin{enumerate}[\bfseries (L1)]
	\item \label{L1_1} For any $\nu \in \PP \left( \T^d \right)$, the function $(x,u) \mapsto L(x, u, \nu)$ is of class $\CC^2$. 
	
	\item \label{L2_1} There is a modulus function $\omega: [0, +\infty) \rightarrow [0, +\infty)$ such that
	$$
	\vert L\left( x,u,\nu_1 \right) - L\left( x,u,\nu_2 \right) \vert \leq \omega \left( d_1 (\nu_1, \nu_2) \right), \quad \forall (x,u) \in \T^d \times \R^n,\,\, \forall \nu_1, \nu_2 \in \PP ( \T^d).
	$$
	See Appendix \ref{assumtions and definitions} and \ref{measure_theory} for details of the modulus function $\omega$ and the Kantorovich-Rubinstein distance $d_1$.
	
	\item \label{L3_1} There is a constant $C_1 \in \R$ such that
	$$
	\left\vert L \left( x,u,\nu \right) \right\vert \leq C_1 \left( 1+ \vert u \vert^{q} \right) , \quad \forall (x,u,\nu) \in \T^d \times \R^n \times \PP ( \T^d).
	$$
	%$$
	%{C_2+ C_3 \left\vert u \right\vert^{p^\prime} \leq L \left( x,u,\nu \right) \leq C_1 \left( 1+ \vert u \vert^{p^\prime} \right) }, \quad \forall (x,u,\nu) \in \T^d \times \R^n \times \PP ( \T^d).
	%$$
	
	\item \label{L4_1} There is a constant $C_2\in \R$ such that
	\begin{equation*}
		\left\vert \frac{\partial L}{\partial x} \left( x,u,\nu \right) \right\vert \leq C_2 \left( 1+ \vert u \vert^{q} \right),  \quad \forall (x,u,\nu) \in \T^d \times \R^n \times \PP ( \T^d).
	\end{equation*}
\end{enumerate}

%\begin{remarks} \label{C_1>C_3}
%It is not difficult to check that the constants in (L\ref{L3_1}) satisfy 
%$$
%C_1 \geq \max\{C_2, C_3\}\,\, \text{and} \,\,\, C_4 \geq \max\{C_5, C_6\}.
%$$
%\end{remarks}

%\medskip

Let $f(x,u,\nu): \T^d \times \R^n \times \PP \left( \T^d \right) \to \R^d$ be a continuous map satisfying the following assumptions. 
\begin{enumerate}[\bfseries (F1)]
	\item \label{f1_1} There is a constant $C >0$ such that
	$$
	\vert f(x,u,\nu) \vert \leq C\left( 1+ \vert u \vert ^p \right), \quad \forall (x,u,\nu) \in \T^d \times \R^n \times \PP ( \T^d).
	$$
	%$$
	%\vert f(x,u,\nu) \vert \leq C\left( 1+ \vert x \vert + \vert u \vert ^p \right), \quad \forall (x,u,\nu) \in \T^d \times \R^n \times \PP ( \T^d).
	%$$
	
	\item \label{f2_1} For any $(u, \nu)\in \R^n \times \PP \left( \T^d \right)$, the map $x \mapsto f(x,u,\nu)$ is Lipschitz continuous and
	\begin{equation*}
		\Lip_x(f) : = \sup_{ \substack{u\in \R^n,\,\, \nu \in \PP \left( \T^d \right) \\ x_1 \neq x_2 \in \T^d}} \frac{\left\vert f \left( x_1, u,\nu \right)-f \left( x_2, u,\nu \right)\right\vert }{\left \vert x_1 - x_2 \right\vert} < +\infty.
	\end{equation*}
	
	\item \label{f3_1} For any $(x,u)  \in \T^d \times \R^n$, the map $\nu \mapsto f(x,u,\nu)$ is Lipschitz continuous and
	\begin{equation*}
		\Lip_{\nu}(f) : = \sup_{ \substack{x \in \T^d ,\,\, u\in \R^n \\ \nu_1 \neq \nu_2 \in \PP \left( \T^d \right)} } \frac{\left\vert f \left( x, u,\nu_1 \right)-f \left( x, u,\nu_2 \right)\right\vert }{d_1 \left( \nu_1, \nu_2 \right)} < +\infty.
	\end{equation*}
\end{enumerate}

\subsection{Main results} \label{main_results}
%Due to the difficulty in locating minimizers within a non-compact set, we will seek compact subsets $\Gamma_T^R \subset \Gamma_T$ and $\PP_{U}^{R} \subset \PP_U$ to consider.
%It is important to note that while the set $\PP \left( \Gamma_T^R \times \PP_{U}^{R} \right)$ is indeed compact, this compactness cannot reveal the relation between the state variable and the measure variable. More precisely, for any $P \in \PP \left( \Gamma_T^R \times \PP_{U}^{R} \right)$ and any $(\gamma, \mu) \in \supp (P)$, there exists a player $i$, of whom the state process is $\gamma$ when the strategy is $\mu$.
%Therefore, for each $m \in \M_r$, which is also a compact set, we will define a compact subset $\PP_R(m) \subset \PP \left( \Gamma_T^R \times \PP_{U}^{R} \right)$ for our specific considerations. The sets $\Gamma_T^R$, $\PP_{U}^{R}$, $M_r$ and $\PP_R(m)$ will be stated as follows.
%
%
%\medskip

%Note that both distribution $m$ and the trajectory $\gamma$ are related to $\mu$.
%Every player aims to find the most appropriate strategy $\mu^* \in \PP_U$ such that the cost function, which is defined by
%$$
%J^m \left( \gamma, \mu \right) = \int_0^T \int_{\R^n} L \left( \gamma(t), u, m(t) \right) \mu(dt,du),
%$$
%attains the minimum.

%define compact set \PP_{U}^{R}
Let $R$ be a positive constant.
%Let $R$ be a constant with $R \geq \frac{(C_1-C_2)T}{C_3}$.
Define the compact subset $\PP_{U}^{R} \subset \PP_U$ by
$$
\PP_{U}^{R} : =\left\{ \mu \in \PP_U \, \middle\vert \,\, \int_0^T \int_{\R^{n}} |u|^{q} \mu( dt, d u) \leq R \right\},
$$
where $q$ is the positive constant defined as in Section \ref{assumptions}.
Define the metric $\tilde{d_1}$ on $\PP_{U}$ by
$$
\tilde{d_1} \left( \mu_1 , \mu_2 \right) := d_1 \left( \frac{\mu_1}{T} , \frac{\mu_2}{T} \right), \quad \forall \mu_1, \mu_2 \in \PP_{U}.
$$
%Note that $\PP_{U}^{R}$ is a compact subset of $\PP_U$ with respect to the $\tilde{d_1}$-topology.

%\medskip

For any $(x, \mu) \in \T^d \times \PP^R_U$ and any $m \in \CC \left( [0,T], \PP \left( \T^d \right) \right)$ such that $m(0)=m_0$, we denote by $\gamma$ the solution of the state equation (\ref{traj_1}). By Lemma \ref{bounded_of_dot(gamma)}, there exists a positive constant $K$, independent of $x$, $\mu$ and $m$, such that
$$
\left\Vert \dot{\gamma} \right\Vert_{L_{q/p} (\LL_{[0,T]})} := \left(\int_0^T \left\vert \dot{\gamma}(t) \right\vert^{q/p} dt \right)^{p/q} \leq K.
$$

%\medskip

%define compact set \M_r
Since $q>p\geq 1$, there exists a positive constant $r>1$ such that $p/q+1/r =1$. Define the subset $\M_r \subset \CC \left( [0,T], \PP \left( \T^d \right) \right)$ by the space of $1/r$-H\"older-continuous maps with H\"older seminorm $K$, i.e.,
$$
\M_r : = \left\{m \in\CC \left( [0,T], \PP ( \T^d ) \right)\, \middle\vert \,\, \sup_{t_1 \neq t_2 \in [0,T]} \frac{d_1 \left( m(t_1), m(t_2) \right)}{\left\vert t_1 - t_2 \right\vert ^{1/r}} \leq K,\,\, m(0)=m_0  \right\}.
$$
Define the metric $\hat{d_1}$ on $\CC \left( [0,T], \PP \left( \T^d \right) \right)$ by
$$
\hat{d_1} \left(m^1,m^2 \right) := \sup_{t \in [0,T]} d_1 \left( m^1(t), m^2(t) \right), \quad \forall m^1,m^2 \in \CC \left( [0,T], \PP ( \T^d ) \right).
$$
By Lemma \ref{M_r_compactness}, $\M_r$ is a compact subset of $\CC \left( [0,T], \PP \left( \T^d \right) \right)$ with respect to the $\hat{d_1}$-topology.

\medskip

%define compact set \Gamma_T^R
Define the subset $\Gamma_T^R \subset \Gamma_T$ by
\begin{align*}
	\Gamma_T^R : = \left\{ \gamma \in  AC \left( [0,T],\T^d \right) \, \right\vert \,\,
	&\exists\; (x,\mu,m) \in \T^d \times \PP_{U}^{R} \times \M_r, \\ &\left. s.t. \, \gamma \,\, \text{is the solution of the state equation} \, (\ref{traj_1}) \right\}.	
\end{align*}
Let $\Gamma_T^R$ also be endowed with the uniform norm $\Vert \cdot \Vert _{\infty}$.
By Lemma \ref{bounded of gamma and gamma^dot_1}, $\Gamma_T^R$ is a compact subset of $\Gamma_T$.  We denote by $\gamma=\gamma \left( \cdot; x, \mu, m \right) \in \Gamma_T^R$ the solution of the state equation (\ref{traj_1}) with the initial value $x \in \T^d$, the measure $\mu \in \PP^R_U$ and the function $m \in \M_r$.

%For each $t \in [0,T]$, define the evaluation map $e_t: \Gamma_T^R \rightarrow \T^d$ by
%$$
%e_t \left(\gamma \right) : = \gamma(t), \quad \forall \gamma \in \Gamma_T^R.
%$$

%\medskip

%define compact set \PP_R(m)
\begin{definition}
	Let $m \in \M_r$. Define $\PP_R(m)$ by the set of probability measures $P \in \PP \left( \Gamma_T^R \times \PP_{U}^{R} \right)$ satisfying the following conditions.
	\begin{enumerate}[(1)]
		\item $e_0 \sharp \left( \pi_1 \sharp P \right) = m_0$.
		
		\item For any $(t,v ) \in [0,T] \times \R^d$ and any open set $N \subset \Gamma_T^R \times \PP_{U}^{R}$, 
		\begin{equation}\label{new_con_1}
			\int_{N} \left \langle v, \gamma(t) - \gamma(0) - \int_0^t \int_{\R^n} f \left( \gamma(s), u,m(s) \right) \mu \left( ds,du \right) \right \rangle P \left( d\gamma, d \mu \right) =0 .
		\end{equation}
	\end{enumerate}
\end{definition}

%\medskip

%the problem we consider
For any $m \in \M_r$ and any $P\in \PP_R(m)$, define the MFC functional by
$$
J\left( m,P \right) : = \int_{\Gamma_T^R \times \PP_U^R} J^m \left( \gamma, \mu \right) P \left( d \gamma, d \mu \right),
$$
where
$$
J^m \left( \gamma, \mu \right) : = \int_0^T \int_{\R^n} L \left( \gamma(t),u,m(t) \right) \mu(dt,du).
$$
Define the associated minimization problem by
$$
R^*(m) : = \argmin_{P \in \PP_R(m)} \left \{J(m, P) \right\}, \quad \forall m \in \M_r.
$$

%\medskip

%definition of relaxed type MFG equilibrium
\begin{definition}
	Given $m_0 \in \PP \left( \T^d \right)$, we say that $P$ is a {\bf relaxed MFC equilibrium} for $m_0$ if
	$$
	P \in R^* \left( \left( e_t \sharp \left( \pi_1 \sharp P \right) \right)_{t \in [0,T]} \right) \,\, \text{and} \,\, \, e_0 \sharp \left( \pi_1 \sharp P \right) = m_0.
	$$
\end{definition}
The relaxed MFC equilibrium is well defined since $\left( e_t \sharp \left( \pi_1 \sharp P \right) \right)_{t \in [0,T]} \in \M_r$ for any $P \in \PP \left( \Gamma_T^R \times \PP^R_U  \right)$. See Remark \ref{well-defined} for details.

%\medskip
\medskip

%main result 1
The first main result is stated as follows.
\begin{theorem} \label{existence_1}
	Assume (L\ref{L1_1})-(L\ref{L4_1}) and (F\ref{f1_1})-(F\ref{f3_1}). There exists at least one relaxed MFC equilibrium. 
\end{theorem}
%Our methods depend on Kakutani's theorem by defining the set valued map by
%$$
%E(\eta) := \left\{ \pi_1 \sharp P \, \middle\vert \,\, P \in R^*\left( \left(e_t \sharp \eta \right)_{t \in [0,T]} \right) \right\}, \quad \forall \eta \in \PP \left( \Gamma_T^R \right).
%$$
%{\color{red}\textit{ Delete or not.} Theorem \ref{existence_1} establishes that any optimization problem \eqref{optimization problem} satisfying our assumptions admits a minimizer when the distribution of all agents is given by $m(t)= e_t \sharp \pi_1 \sharp P$ for some relaxed MFC equilibria $P$. For any individual player, its trajectory $\gamma$ and strategy $\mu$ must satisfy \eqref{222_1}.
%To the best of our knowledge, this theorem is the first step towards understanding first-order non-convex MFC problems.}
%
\medskip

For any $(x,m) \in \T^d \times \M_r$ and any $u_x: [0,T] \rightarrow \R^n$ such that $\int_0^T \left\vert u_x(t) \right\vert^q dt \leq R$, denote by $\gamma_x^{u_x} : = \gamma_x^{u_x} \left( \cdot ; x, \LL_{[0,T]} \otimes \delta_{u_x(t)}, m \right)$ the solution of the state equation (\ref{traj_1}). Define a probability measure $\eta^u \in \PP \left( \Gamma_T^R \right)$ by
\begin{equation} \label{def_of_eta^u}
	\eta^u(A) : = \int_{\T^d} \delta_{ \left\{\gamma_x^{u_x}\right\} }(A) m_0(dx), \quad \forall A \in \mathcal{B} \left( \Gamma_T^R \right).
\end{equation}
By Lemma \ref{pre1}, $\eta^u$ is the unique measure such that the probability measure $P^u := \eta^u \otimes \delta_{\left\{\LL_{[0,T]} \otimes \delta_{u_{\gamma(0)}(t)} \right\}}$ belongs to the set $\PP_R(m)$.

%\medskip

%definition of strict relaxed MFG equilibrium
\begin{definition}
	Given $m_0 \in \PP \left( \T^d \right)$, we say that $P$ is a {\bf strict relaxed MFC equilibrium} for $m_0$ if it satisfies:
	\begin{itemize}
		\item[($i$)] $P$ is a relaxed MFC equilibrium for $m_0$.
		\item[($ii$)] There exist maps $u : [0,T] \times \T^d  \rightarrow \R^n$ such that $\int_0^T \left\vert u(t, x) \right\vert^q dt \leq R$ and 
		$$
		P = \eta^u \otimes \delta_{ \left\{\LL_{[0,T]} \otimes \delta_{u(t, x)}\right\} },
		$$
		where $\eta^u$ is defined as in (\ref{def_of_eta^u}).
	\end{itemize}
\end{definition}

%\medskip

The second main result is stated as follows and, as a consequence, we immediately obtain that Theorem \ref{existence_1} encompasses the classical existence result for MFC equilibria.
\begin{theorem} \label{relation_with_convex_case}
	Let $m \in \M_r$ and $P \in R^* \left( m \right)$. For any $(t,x) \in [0,T] \times \T^d$, the set
	$$
	\mathcal{L} \left(t,x \right) : = \left\{ \left( \lambda,w, v \right) \in \R^{d+2} \, \middle \vert \,\, \exists\; u \in \R^n, s.t.\, \lambda \geq L \left(x,u,m(t) \right), \, w \geq \vert u \vert^q, \, v = f \left( x,u,m(t) \right) \right\}
	$$
	is convex. Then the following statements hold.
	\begin{enumerate}[(1)]
		\item For any $x \in \T^d$ there exist maps $u^*_x: [0,T] \rightarrow \R^n$ such that the probability measure $P_0 := \eta^{u^*} \otimes \delta_{\left\{ \LL_{[0,T]} \otimes \delta_{u^*_{\gamma(0)}(t)}\right\} } \in R^*(m)$, where $\left\{ u^*_x \right\}_{x \in \T^d}$ is the optimal control for
		$$
		\inf_{u} \int_{\T^d} \int_0^T L \left( \gamma_x^{u}(t), u(t), m(t) \right) dt\, m_0(dx),
		$$
		subject to the state equation
		\begin{equation} \label{thm4.2_need}
			\begin{cases}
				\dot{\gamma}^{u}(t) =f \left( \gamma^{u}(t), u(t), m(t) \right), & t \in [0, T]
				\\
				\gamma^{u}(0) = x, & x \in \T^d
			\end{cases}
		\end{equation}
		where the infimum is taken over all the maps $u_x: [0,T] \rightarrow \R^n$ such that $\int_0^T \left \vert u_x(t) \right\vert^q dt \leq R$ for any $x \in \T^d$.
		\item There exists at least one strict relaxed MFC equilibrium. 
	\end{enumerate}
\end{theorem}
%
%{\color{red} Delete or not. This theorem implies that the relaxation method can also be applied to the classical MFC problem, whose value function is convex in the control.}

\begin{remarks}\em
	In synthesis, according to Theorem \ref{existence_1} and Theorem \ref{relation_with_convex_case}, we have that all the information of the neural network architecture described in \eqref{RESNET1} and \eqref{RESNET2} is captured by the relaxed MFC equilibrium. In particular, it is on the second marginal of a relaxed MFC equilibrium that we can find the parameters of the neural network that minimize the suitable loss function with respect to the analyzed data. 
\end{remarks}

\medskip
%structure of this paper
\noindent{\bf Structure of this paper.} In Section \ref{relaxedMFG}, we establish the existence of relaxed MFC equilibria. In Section \ref{strictMFG}, we demonstrate that these relaxed equilibria possess some suitable structures and that under the classical convexity assumption, the new equilibria coincide with the classical MFC equilibria. Moreover, we introduce some notations in Appendix \ref{assumtions and definitions}, recall well-known results from measure theory which will be used throughout this paper in Appendix \ref{measure_theory} and complete the proof of Proposition \ref{3.7_1_1} in Appendix \ref{Completion of the Proof of Proposition 3.7_1_1}.

\medskip

\section{Existence of relaxed MFC equilibria}\label{relaxedMFG}

\subsection{Proof of the properties of $\PP_R(m)$}
\begin{lemma} \label{bounded_of_dot(gamma)}
There exists a constant $K>0$ such that for any $(x,\mu)\in \T^d \times \PP^R_U$ and any $m \in \CC \left( [0,T], \PP \left( \T^d \right) \right)$, the solution $\gamma$ of the state equation (\ref{traj_1}) satisfies
$$
\left\Vert \dot{\gamma} \right\Vert_{L_{q/p} \left( \LL_{[0,T]} \right) }  := \left(\int_0^T \left\vert \dot{\gamma}(t) \right\vert^{q/p} dt \right)^{p/q} \leq K.
$$
\end{lemma}
\begin{proof}
%By H\"older's inequality, for any $\mu \in \PP^R_U$ and any $t \in [0,T]$, we have
%$$
%\int_0^t \int_{\R^n} \left\vert u \right\vert^p \mu (ds,du) \leq \left( \int_0^t \int_{\R^n} \left\vert u \right\vert^q \mu (ds,du) \right)^{\frac{p}{q}}  T^{\frac{1}{r}} \leq R^{\frac{p}{q}} T^{\frac{1}{r}},
%$$
%where the constants $p$, $q$ and $r$ are defined as in Section \ref{assumptions}.
%Define a constant $R^{\prime} : = R^{\frac{p}{q}} T^{\frac{1}{r}}$.
%Since $\T^d$ is compact, there exists a constant $\hat{C}>0$ such that $\sup_{x \in \T^d} \vert x \vert \leq \hat{C}$.

Since $\gamma$ is the solution of the state equation (\ref{traj_1}), we have
\begin{align*}
	\int_0^T \left\vert \dot{\gamma}(t) \right\vert^{\frac{q}{p}} dt
	& = \int_0^T \left\vert \int_{\R^n} f \left( \gamma(t), u , m(t) \right) \mu_t(du) \right\vert ^{\frac{q}{p}} dt \\
	& \leq \int_0^T \int_{\R^n} \left\vert f \left( \gamma(t), u , m(t) \right) \right\vert ^{\frac{q}{p}} \mu_t(du) dt \\
	& \leq \int_0^T \int_{\R^n} C^{\frac{q}{p}} \left( 1+ \vert u \vert ^p \right)^{\frac{q}{p}} \mu_t(du) dt \\
%	& \leq \int_0^T \int_{\R^n} C^{\frac{q}{p}} \left( 1+\hat{C}+ \vert u \vert ^p \right)^{\frac{q}{p}} \mu_t(du) dt \\
	& \leq 2^{\frac{q}{p} -1} T C^{\frac{q}{p}} + 2^{\frac{q}{p} -1} C^{\frac{q}{p}}R.
\end{align*}
Here, the second line holds by Jensen's inequality, the third one holds by (F\ref{f1_1}) and the last one holds by the convexity of $x^{q/p}$.
%Define the function $h:\R^+ \rightarrow \R^+$ by
%$$
%h(x) := x^{q/p}, \quad \forall x \in \R^+.
%$$
%Since the function $h$ is convex, we have
%\begin{equation} \label{123}
%\left( \frac{a+b}{2} \right)^{\frac{q}{p}} \leq \frac{a^{\frac{q}{p}} + b^{\frac{q}{p}}}{2}, \quad \forall a,b \in \R^+.
%\end{equation}
Let the constant $K$ equal to $\left( 2^{\frac{q}{p} -1} T C^{\frac{q}{p}} + 2^{\frac{q}{p} -1} C^{\frac{q}{p}}R \right)^{\frac{p}{q}}$. The proof is complete.
\end{proof}

\begin{remarks} \label{bounded_of_f}
	For any $(t,x,\mu) \in [0,T] \times \T^d \times \PP_{U}^{R}$ and any $m \in \CC \left( [0,T], \PP(\T^d) \right)$, denote by $\gamma$ the related solution of the state equation (\ref{traj_1}). We have
	$$
	\int_0^t \int_{\R^n} \left\vert u \right\vert^p \mu (ds,du) \leq \left( \int_0^t \int_{\R^n} \left\vert u \right\vert^q \mu (ds,du) \right)^{\frac{p}{q}}  T^{\frac{1}{r}} \leq R^{\frac{p}{q}} T^{\frac{1}{r}}
	$$
	%,\quad \int_0^t \int_{\R^n} \left\vert u \right\vert^{p^\prime} \mu (ds,du) \leq R^{\frac{p^\prime}{q}} T^{1-\frac{p^\prime}{q}}
	and
	$$
	\int_0^t \int_{\R^n} \vert f \left( \gamma(s), u , m(s) \right)\vert \mu (ds, d u) \leq CT + CR^{\frac{p}{q}} T^{\frac{1}{r}}.
	$$
\end{remarks}

%\medskip

%compact of M_r
\begin{lemma} \label{M_r_compactness}
	$\M_r$ is a compact subset of $\CC \left( [0,T], \PP \left( \T^d \right) \right)$ with respect to the $\hat{d_1}$-topology, where $\hat{d_1}$ is defined as in Section \ref{main_results}.
\end{lemma}
\begin{proof}
Since $\T^d$ is compact, $\PP \left( \T^d \right)$ is also compact with respect to the $d_1$-topology. Thus, $\M_r$ is uniformly bounded. By the construction of $\M_r$, it is uniformly equi-H\"older-continuous. By Ascoli-Arzel\`a theorem, $\M_r$ is relatively compact. Obviously, $\M_r$ is closed under $\hat{d_1}$-topology. We conclude that $\M_r$ is compact.
%For any sequence $\left\{m^i\right\}_{i=1}^{\infty} \subset \M_r$ that converges to some $m \in \CC \left( [0,T], \PP \left( \T^d \right) \right)$ with respect to the $\hat{d_1}$-topology,
%it is clear that 
%$$
%\lim_{i \rightarrow \infty} d_1 \left( m^i (t), m(t) \right) =0, \quad \forall t \in [0,T].
%$$
%for any $t_1 \neq t_2 \in [0,T]$, we have
%$$
%d_1 \left( m (t_1), m(t_2) \right) = d_1 \left( m^i (t_1), m(t_1) \right)+d_1 \left( m^i (t_1), m^i(t_2) \right)+d_1 \left( m^i (t_2), m(t_2) \right).
%$$
%Let $i$ tend to infinity. We obtain that
%$$
%d_1 \left( m (t_1), m(t_2) \right) \leq K \left\vert t_1 - t_2 \right\vert ^{1/r}.
%$$
%Since $\lim_{i \rightarrow \infty} d_1 \left( m^i (0), m(0) \right) = d_1 \left( m_0, m(0) \right) =0$, it follows that $m(0) = m_0$. Thus, we have $m \in \M_r$.
%Therefore, $\M_r$ is compact with respect to the $\hat{d_1}$-topology.
\end{proof}

%\medskip

%compact of \Gamma_T^R
\begin{lemma}\label{bounded of gamma and gamma^dot_1}
	$\Gamma_T^R$ is a compact subset of $\Gamma_T$ with respect to the uniform norm $\Vert \cdot \Vert _{\infty}$.	
\end{lemma}
\begin{proof}
Since $\T^d$ is compact, $\Gamma_T^R$ is uniformly bounded. 
By Lemma \ref{bounded_of_dot(gamma)}, for any $s \leq t \in [0,T]$, we have
\begin{align*}
	\left\vert \gamma(t)-\gamma(s) \right\vert 
	& = \left\vert \int_s^t \dot{\gamma}(\tau) d \tau \right\vert \leq \int_s^t \left\vert \dot{\gamma}(\tau) \right\vert d \tau \\
	& \leq \left( \int_s^t \left\vert \dot{\gamma}(\tau) \right\vert^{\frac{q}{p}} d \tau \right)^{\frac{p}{q}} \cdot \left( t-s \right) ^{\frac{1}{r}} \leq K \cdot \left( t-s \right) ^{\frac{1}{r}},
	\end{align*}
implying that $\Gamma_T^R$ is uniformly equi-H\"older-continuous.
Thus, $\Gamma_T^R$ is relatively compact by Ascoli-Arzel\`a theorem.

Consider a sequence $\left\{\gamma_i\right\}_{i=1}^{\infty} \subset \Gamma_T^R$ converging to some $\tilde{\gamma}$ in the uniform norm $\left\Vert \cdot \right\Vert_{\infty}$.
There exist sequences $\left\{x_i\right\}_{i=1}^{\infty} \subset \T^d$, $\left\{m^i\right\}_{i=1}^{\infty} \subset \M_r$ and $\left\{\mu^i\right\}_{i=1}^{\infty} \subset \PP_U^R$ such that
$$
\gamma_i(t) = x_i+ \int_0^t \int_{\R^n} f \left( \gamma_i(s), u ,m^i(s) \right) \mu^i_s(du)ds, \quad \forall t \in [0,T].
$$
Since the sets $\T^d$, $\M_r$ and $\PP_U^R$ are compact, there exist $x \in \T^d$, $m \in \M_r$ and $\mu \in \PP^R_U$ such that the sequences $\left\{x_i\right\}_{i=1}^{\infty}$, $\left\{m^i\right\}_{i=1}^{\infty}$ and $\left\{\mu^i\right\}_{i=1}^{\infty}$ converge to $x$, $m$ and $\mu$ respectively.
Define $\gamma(t): = x + \int_0^t \int_{\R^n} f \left( \gamma(s), u, m(s) \right) \mu_s(du)ds$. 
We claim that $\gamma = \tilde{\gamma}$. Indeed,
\begin{align*}
\left\vert \gamma_i(t) - \gamma(t) \right\vert
\leq & \left\vert x_i - x \right\vert + \int_0^t \int_{\R^n} \left\vert f \left( \gamma_i(s), u, m^i(s) \right) - f \left(\gamma(s), u, m(s) \right) \right\vert \mu^i_s(du)ds \\
& + \int_0^t \int_{\R^n} \left\vert f \left(\gamma(s), u, m(s) \right) \right\vert (\mu^i_s - \mu_s) (du)ds \\
\triangleq & I_1+I_2+I_3.
\end{align*}
First of all, $\lim_{i \rightarrow \infty} I_1 =0$ obviously. 
By Proposition \ref{D-L A.1}, we have $\lim_{i \rightarrow \infty} I_3 =0$. By (F\ref{f2_1}) and (F\ref{f3_1}),
\begin{align*}
I_2
\leq & \int_0^t \int_{\R^n} \left\vert f \left( \gamma_i(s), u, m^i(s) \right) - f \left(\gamma(s), u, m^i(s) \right) \right\vert \mu^i_s(du)ds \\
& + \int_0^t \int_{\R^n} \left\vert f \left( \gamma(s), u, m^i(s) \right) - f \left(\gamma(s), u, m(s) \right) \right\vert \mu^i_s(du)ds \\
%\leq & \int_0^t \Lip_x(f) \left\vert \gamma_i(s) - \gamma(s) \right\vert  + \Lip_{\nu}(f)\, d_1 \left( m^i(s) , m(s) \right) ds \\
\leq & \Lip_x(f) \int_0^t \left\vert \gamma_i(s) - \gamma(s) \right\vert ds + \Lip_{\nu}(f) T \hat{d}_1 \left( m^i , m \right).
\end{align*}
Since $\hat{d}_1 \left(m^i,m \right) \rightarrow 0$ as $i \rightarrow \infty$, there exists a sequence of constants $\left\{c_i\right\}_{i=1}^{\infty}$ such that $\lim_{i \rightarrow \infty} c_i =0$ and 
$$
I_1+ I_3 + \Lip_{\nu}(f) T \hat{d}_1 \left( m^i , m \right) \leq c_i.
$$
By Gronwall's inequality, 
$$
\left\vert \gamma_i(t) - \gamma(t) \right\vert \leq c_i e^{\Lip_x(f) t} \leq c_i e^{\Lip_x(f) T}.
$$
In conclusion, we have $\lim_{i \rightarrow \infty} \left\Vert \gamma_i - \gamma \right\Vert_{\infty} = 0$.
Furthermore, for any $i \in \N$, we obtain that
$$
\left\Vert \gamma - \tilde{\gamma} \right\Vert_{\infty} \leq \left\Vert \gamma - \gamma_i \right\Vert_{\infty} +\left\Vert \gamma_i - \tilde{\gamma} \right\Vert_{\infty}.
$$
Let $i$ tend to infinity. 
The limit $\tilde{\gamma}$ of the sequence $\left\{\gamma_i\right\}_{i=1}^{\infty}$ is actually the curve $\gamma$ we defined above. Therefore, $\Gamma_T^R$ is compact.
\end{proof}

%\medskip

%some properties of P_R(m)
\begin{proposition} \label{properties_PR_1}
	For each $m \in \M_r$, the set $\PP_R(m)$ has the following properties.
	\begin{enumerate}[(1)]
		\item $\PP_R(m)$ is compact with respect to the narrowly convergence.
%		\item $\PP_R(m)$ is compact with respect to the $d_1$-topology.
		\item $\PP_R(m)$ is non-empty.
		\item For any $P \in \PP \left( \Gamma_T^R \times \PP_{U}^{R} \right)$, the probability measure $P$ satisfies the condition (\ref{new_con_1}) if and only if for each state-measure pair $\left( \gamma, \mu \right) \in \supp(P)$, the pair satisfies 
		\begin{equation} \label{222}
			\dot{\gamma}(t) = \int_{\R^n} f \left( \gamma(t), u, m(t) \right) \mu_t(du).
		\end{equation}
	\end{enumerate}
\end{proposition}
\begin{proof}
(1) Since $\Gamma_T^R$ and $\PP_{U}^{R}$ are compact, the set $\PP \left( \Gamma_T^R \times \PP_{U}^{R} \right)$ is also compact with respect to the narrowly convergence.
Let $\left\{ P_i \right\}_{i =1}^{\infty} \subset \PP_R(m)$ such that $P_i$ converges to some $P \in \PP \left( \Gamma_T^R \times \PP_{U}^{R} \right)$ narrowly. 
%For any function $g \in \CC_b \left( \T^d \right)$, we have 
%\begin{align*}
%	\int_{\T^d} g(x) m_0(dx)
%	& = \lim_{i \rightarrow \infty} \int_{\T^d} g(x) e_0 \sharp \left( \pi_1 \sharp P_i \right) (dx) \\
%	& = \lim_{i \rightarrow \infty}\int_{\Gamma_T^R \times \PP_{U}^{R}} g \left( \gamma(0) \right) P_i \left( d\gamma, d\mu \right) \\
%	& = \int_{\Gamma_T^R \times \PP_{U}^{R}} g \left( \gamma(0) \right) P \left( d\gamma, d\mu \right) \\ 
%	& = \int_{\T^d} g(x) e_0 \sharp \left( \pi_1 \sharp P \right) (dx),
%\end{align*}
%which indicates that $e_0 \sharp \left( \pi_1 \sharp P \right) = m_0$.
Then it is clear that $e_0 \sharp \left( \pi_1 \sharp P \right) = m_0$.

For each $(t,v) \in [0,T] \times \R^d$, define the function $\kappa: \Gamma_T^R \times \PP_{U}^{R} \rightarrow \R$ by
\begin{equation} \label{equation in property P_R(m)}
\left(\gamma, \mu \right) \mapsto \left \langle v, \gamma(t) - \gamma(0) - \int_0^t \int_{\R^n} f \left( \gamma(s), u,m(s) \right) \mu \left( ds,du \right) \right \rangle.
\end{equation}
We claim that the function $\kappa$ is bounded and continuous.
It is obviously that the function 
\begin{equation*}\label{function1}
\left( \gamma, \mu \right) \mapsto \langle v, \gamma(t)- \gamma(0) \rangle
\end{equation*}
is bounded and continuous. Additionally, by Remark \ref{bounded_of_f}, the function
\begin{equation} \label{function2}
\left( \gamma, \mu \right) \mapsto \left\langle v, \int_0^t \int_{\R^n} f \left( \gamma(s) , u,m(s) \right) \mu (ds, du) \right\rangle
\end{equation}
is bounded as well.
On the other hand, let $\gamma_i \rightarrow \gamma$ in $\Gamma_T^R$ and $\mu^i \rightarrow \mu$ in $\PP_{U}^{R}$.
%It is obviously that the function (\ref{function1}) is continuous.
Consider
\begin{align*}
A_1 : = &\left\vert \int_0^t \int_{\R^n} f \left( \gamma_i(s) , u ,m(s) \right) \mu^i (ds,du) - \int_0^t \int_{\R^n} f \left( \gamma(s) , u ,m(s) \right) \mu (ds,du) \right\vert \\
\leq & \int_0^t \int_{\R^n} \left\vert f \left( \gamma_i(s) , u ,m(s)\right) -   f \left( \gamma(s) , u,m(s) \right)\right \vert \mu^i (du,ds) \\
& + \left\vert \int_0^t  \int_{\R^n} f \left( \gamma(s) , u ,m(s) \right) \mu^i (ds,du) -  \int_0^t \int_{\R^n} f \left( \gamma(s) , u ,m(s) \right) \mu (ds,du) \right \vert \\
\triangleq & A_2 + A_3.
\end{align*}
By (F\ref{f2_1}), 
$$
\lim_{i \rightarrow \infty} A_2 \leq \lim_{i \rightarrow \infty}\int_0^t \int_{\R^n} \Lip_x(f) \left\Vert \gamma_i - \gamma \right\Vert _{\infty} \mu^i (du,ds) \leq \lim_{i \rightarrow \infty} \Lip_x(f) T \left\Vert \gamma_i - \gamma \right\Vert _{\infty} = 0 .
$$
By (F\ref{f1_1}) and Proposition \ref{D-L A.1}, $\lim_{i \rightarrow \infty} A_3 = 0$. We conclude that the function $\kappa$ is both bounded and continuous.

For any open set $N \subset \Gamma_T^R \times \PP_{U}^{R}$, we have
\begin{align*}
& \left\vert \int_{N} \left\langle v, \gamma(t) - \gamma(0) - \int_0^t \int_{\R^n} f \left( \gamma(s), u,m(s) \right) \mu \left( ds,du \right) \right\rangle P \left( d\gamma, d \mu \right) \right\vert \\
%\leq & \int_{\Gamma_T^R \times \PP_{U}^{R}} \mathbf{1}_N \left( \gamma, \mu \right) \left\vert \left \langle v, \gamma(t) - \gamma(0) - \int_0^t \int_{\R^n} f \left( \gamma(s), u,m(s) \right) \mu \left( ds,du \right) \right \rangle \right\vert P \left( d\gamma, d \mu \right) \\
\leq & \liminf_{i \rightarrow \infty} \int_{\Gamma_T^R \times \PP_{U}^{R}} \mathbf{1}_N \left( \gamma, \mu \right) \left\vert \left \langle v, \gamma(t) - \gamma(0) - \int_0^t \int_{\R^n} f \left( \gamma(s), u,m(s) \right) \mu \left( ds,du \right) \right \rangle \right\vert P_i \left( d\gamma, d \mu \right) \\
= & \liminf_{i \rightarrow \infty} \int_{N^+} \left \langle v, \gamma(t) - \gamma(0) - \int_0^t \int_{\R^n} f \left( \gamma(s), u,m(s) \right) \mu \left( ds,du \right) \right \rangle P_i \left( d\gamma, d \mu \right) \\
& - \int_{N^-} \left \langle v, \gamma(t) - \gamma(0) - \int_0^t \int_{\R^n} f \left( \gamma(s), u,m(s) \right) \mu \left( ds,du \right) \right \rangle P_i \left( d\gamma, d \mu \right) =0,
\end{align*}
where the sets $N^+$ and $N^-$ are defined by
$$
N^+ := \left\{ \left( \gamma, \mu \right) \in N \, \middle\vert \,\, \kappa \left( \gamma, \mu \right) > 0 \right\},
$$
$$
N^- := \left\{ \left( \gamma, \mu \right) \in N \, \middle\vert \,\, \kappa \left( \gamma, \mu \right) < 0 \right\}.
$$
Since the function $\kappa$ is continuous, both $N^+$ and $N^-$ are open sets.
In conclusion, $P \in \PP_R(m)$, and $\PP_R(m)$ is closed. Furthermore, since $\PP_R(m) \subset \PP \left( \Gamma_T^R \times \PP_{U}^{R} \right)$ and $\PP \left( \Gamma_T^R \times \PP_{U}^{R} \right)$ is compact, the set $\PP_R(m)$ is also compact with respect to the narrowly convergence.

%\medskip

%(2) This point is a direct consequence of the point (1).

\medskip

(2) Let $\mu_0 = \LL_{[0,T]} \otimes \delta_{\{0\}} \in \PP^R_U$. The state equation (\ref{traj_1}) is transformed into an autonomous ordinary differential equation
\begin{equation} \label{need in (3)}
	\begin{cases}
		\dot\gamma(t) =  f(\gamma(t), 0 , m(t)) ,
		\\
		\gamma(0) = x.
	\end{cases}
\end{equation}
For any $x \in \T^d$, there exists a unique solution $\gamma_x (t) = \gamma_x (t; x, \mu_0, m)$ of (\ref{need in (3)}) by (F\ref{f2_1}). Define the map $p: \T^d \rightarrow \Gamma_T^R$ by 
$$
x \mapsto p(x) = \gamma_x.
$$ 
Define $P := p \sharp m_0  \otimes \delta_{\{\mu_0\}} \in \PP \left( \Gamma_T^R \times \PP_{U}^{R} \right)$.
%For any function $g \in \CC_b \left( \T^d \right)$, we have
%\begin{align*}
%	\int_{\T^d} g(x) e_0 \sharp \left( \pi_1 \sharp P \right) (dx) 
%	& = \int_{\Gamma_T^R} g \left( \gamma(0) \right) \pi_1 \sharp P \left(d \gamma \right) \\
%	& = \int_{\Gamma_T^R} g \left( \gamma(0) \right) p \sharp m_0 \left( d \gamma \right) \\
%	& = \int_{\T^d} g \left( p(x)(0) \right) m_0 \left( d x \right) \\
%	& = \int_{\T^d} g(x) m_0 \left(d x \right),
%\end{align*}
%which indicates that $e_0 \sharp \left( \pi_1 \sharp P \right) = m_0$.
It is clear that $P \in \PP_R(m)$.

%For any $(t,v) \in [0,T] \times \R^d$ and any open set $N \subset \Gamma_T^R \times \PP_{U}^{R}$,
%\begin{align*}
%	& \int_{N} \left \langle v, \gamma(t) - \gamma(0) - \int_0^t \int_{\R^n} f \left( \gamma(s), u,m(s) \right) \mu \left( ds,du \right) \right \rangle P \left( d\gamma, d \mu \right) \\
%	= & \int_{\pi_1(N)} \left \langle v, \gamma(t) - \gamma(0) - \int_0^t f \left( \gamma(s), 0,m(s) \right) ds \right \rangle \, p \sharp m_0 \left( d \gamma \right) \\
%	= & \int_{p^{-1} \left(\pi_1(N) \right)} \left \langle v, p(x)(t) - p(x)(0) - \int_0^t f \left(p(x)(s), 0 ,m(s)\right) ds \right \rangle\, m_0 (dx) \\
%	= & \int_{p^{-1} \left(\pi_1(N) \right)} \left \langle v, p(x)(t) - p(x)(0) - \int_0^t \frac{\partial p(x)(s)}{\partial s} ds \right \rangle\, m_0(dx) = 0,
%\end{align*}
%where the set $p^{-1} \left(\pi_1(N) \right)$ is defined by
%$$
%p^{-1} \left(\pi_1(N) \right): = \left\{ x\in \T^d \, \middle\vert \,\,p(x) \in \pi_1(N) \right\}.
%$$
%Above all, we have $P \in \PP_R(m)$. Therefore, $\PP_R(m)$ is non-empty.

\medskip

(3) Let $P$ be the probability measure that satisfies the condition (\ref{new_con_1}). Suppose there exist $\left( \gamma_0, \mu^0 \right) \in \supp (P)$ and $t_0 \in [0,T]$ such that
$$
\gamma_0 \left( t_0 \right) - \gamma_0(0) \neq \int_0^{t_0} \int_{\R^n} f \left( \gamma_0(s), u,m(s) \right) \mu^0 \left( ds,du \right).
$$
Without loss of generality, assume there exists $v_0 \in \R^d$ such that
$$
\left\langle v_0, \gamma_0 \left( t_0 \right) - \gamma_0(0) - \int_0^{t_0} \int_{\R^n} f \left( \gamma_0(s), u,m(s) \right) \mu^0 \left( ds,du \right) \right\rangle >0.
$$
Similar to the proof of the continuity of $\kappa$, the map
$$
\left( \gamma, \mu \right) \mapsto \left\langle v_0, \gamma \left( t_0 \right) - \gamma(0) - \int_0^{t_0} \int_{\R^n} f \left( \gamma(s), u ,m(s) \right) \mu \left( ds,du \right) \right\rangle
$$
is also continuous.
Hence, there exists an open neighborhood $N$ of the point $\left( \gamma_0, \mu^0 \right)$ such that $P (N) > 0$ and
$$
\int_N \left\langle v_0, \gamma \left( t_0 \right) - \gamma(0) - \int_0^{t_0} \int_{\R^n} f \left( \gamma(s), u,m(s) \right) \mu \left( ds,du \right) \right\rangle P \left( d\gamma, d\mu \right) >0,
$$
which contradicts (\ref{new_con_1}).
Therefore, we have 
$$
\gamma \left( t \right) - \gamma(0) = \int_0^{t} \int_{\R^n} f \left( \gamma (s), u,m(s) \right) \mu \left( ds,du \right), \quad \forall \left( \gamma, \mu \right) \in \supp(P), \forall t \in [0,T].
$$
Taking the derivative with respect to $t$ on both sides gives us (\ref{222}).

The other point can be easily proved by definition.
\end{proof}

%\medskip
%\medskip

%%%%%%%%%%%%%%%%%%%%%%%%%%%%%%%
%Section 3.2
\subsection{Proof of the existence of relaxed MFC equilibria} \label{existence_relaxed_MFG_Solution}
In this section, we prove the existence of relaxed MFC equilibria by Kakutani's theorem.

\medskip

%Define set-valued map
Define a subset $\PP_0 \left( \Gamma_T^R \right) \subset \PP\left( \Gamma_T^R \right)$ by
$$
\PP_0 \left( \Gamma_T^R \right) := \left\{ \eta \in \PP\left( \Gamma_T^R \right)\, \middle\vert \,\, e_0 \sharp \eta = m_0 \right\}.
$$
Obviously, $\PP_0 \left( \Gamma_T^R \right)$ is compact with respect to the $d_1$-topology.
Define the set-valued map by
$$
E: \left( \PP_0 \left( \Gamma_T^R \right), d_1 \right) \rightrightarrows \left( \PP_0 \left( \Gamma_T^R \right), d_1 \right), \quad \eta \mapsto E(\eta),
$$
where
$$
E(\eta) := \left\{ \pi_1 \sharp P \, \middle\vert \,\, P \in R^*\left( \left(e_t \sharp \eta \right)_{t \in [0,T]} \right) \right\}.
$$

%well-defined of relaxed MFG equilibrium and set-valued map
\begin{remarks} \label{well-defined}
	Let $\eta \in \PP_0 \left( \Gamma_T^R \right)$. Define the map $m:[0,T] \rightarrow \PP \left( \T^d \right)$ by $m(t) = e_t \sharp \eta$ for any $t \in [0,T]$. It is clear that $m \in \M_r$.
	More precisely, it is clear that $m(0) = e_0 \sharp \eta = m_0$.
	For any $t_1 \neq t_2 \in [0,T]$ and any 1-Lipschitz function $g: \T^d \rightarrow \R$, we obtain that
	\begin{align*} 
		&  \left\vert \int_{\T^d} g(x) e_{t_1} \sharp \eta(dx) - \int_{\T^d} g(x) e_{t_2} \sharp \eta(dx) \right\vert \\
		\leq & \int_{\Gamma_T^R} \left\vert g\left( \gamma(t_1) \right) - g\left( \gamma(t_2) \right) \right\vert \eta \left( d\gamma \right) \\
		\leq & \int_{\Gamma_T^R} \left\vert  \gamma(t_1) - \gamma(t_2) \right\vert \eta \left( d\gamma \right)
		\leq K \left\vert t_1 - t_2 \right\vert ^{\frac{1}{r}},
	\end{align*}
	where the last inequality comes from Lemma \ref{bounded_of_dot(gamma)}.
	By the arbitrariness of the function $g$, we have
	$$
	d_1 \left( m(t_1), m(t_2) \right) \leq K \left\vert t_1 - t_2 \right\vert ^{\frac{1}{r}}.
	$$
	Thus, $m \in \M_r$. Therefore, the set-valued map $E$ is well-defined.
\end{remarks}

%\medskip

%J^m(\gamma, \mu) is continuous
\begin{lemma} \label{Jm_cont_1}
	For each $m \in \M_r$,
	\begin{enumerate}[(1)]
		\item the family of functions $\left\{ J^m \left( \cdot, \mu \right): \Gamma_T^R \rightarrow \R\, \middle\vert \,\, \mu \in \PP_U^R \right\}$ is uniformly equi-continuous.
		\item Let $\gamma \in \Gamma_T^R$. If $\mu^i$ converges to $\mu$ narrowly, then
		\begin{equation}\label{J_m lower semi-cont}
			\liminf_{i \rightarrow \infty} J^m \left( \gamma, \mu^i \right) \geq J^m \left( \gamma, \mu \right).
		\end{equation}
	\end{enumerate}
\end{lemma}
\begin{proof}
(1) Fix a sequence $\left\{ \gamma_i \right\} _{i=1}^{\infty}$ that converges to some $\gamma$ in $\Gamma_T^R$. For any $\mu \in \PP_U^R$, we have
\begin{align*}
	& \left\vert J^m \left( \gamma_i, \mu \right) -J^m \left( \gamma, \mu \right) \right \vert
	\leq \int_0^T \int_{\R^n} \left\vert L \left( \gamma_i(t),u,m(t) \right) - L \left( \gamma(t),u,m(t) \right) \right\vert \mu(dt,d u)  \\
	& \leq \left\Vert \gamma_i - \gamma \right\Vert_{\infty} \int_0^T \int_{\R^n} \int_0^1 \left\vert \frac{\partial L}{\partial x} \left( \gamma(t) + \lambda \left( \gamma_i(t)-\gamma(t) \right) , u ,m(t) \right) \right\vert d \lambda \, \mu(dt,du).
\end{align*}
By (L\ref{L4_1}), the proof is completed since we obtain that
\begin{equation*}
	\int_0^T \int_{\R^n} \int_0^1 \left \vert \frac{\partial L}{\partial x} \left( \gamma(t) + \lambda \left( \gamma_i(t)-\gamma(t) \right) , u ,m(t) \right) \right\vert d \lambda \, \mu(dt,du) \leq C_2 \left(T+R \right).
\end{equation*}	
%which completes the proof.

%\medskip

(2) For any $\varepsilon >0$, we obtain that
\begin{align*}
	\int_0^T \int_{\R^n} L \left( \gamma(t) , u , m(t) \right) \mu^i (dt,du)
	& = \int_0^T \int_{\R^n} \frac{L \left( \gamma(t) , u , m(t) \right)}{1 + \varepsilon \vert u \vert ^{q}} \left( 1 + \varepsilon \vert u \vert ^{q} \right) \mu^i (dt,du) \\
	& \geq \int_0^T \int_{\R^n} \frac{L \left( \gamma(t) , u , m(t) \right)}{1 + \varepsilon \vert u \vert ^{q}} \mu^i (dt,du).
\end{align*}
By (L\ref{L1_1}) and (L\ref{L3_1}), the map
$$
(t,u) \mapsto \frac{L \left( \gamma(t) , u , m(t) \right)}{1 + \varepsilon \vert u \vert ^{q}}
$$
is bounded and continuous. Thus,
$$
\liminf_{n \rightarrow \infty} \int_0^T \int_{\R^n} L \left( \gamma(t) , u , m(t) \right) \mu^i (dt,du) \geq \int_0^T \int_{\R^n} \frac{L \left( \gamma(t) , u , m(t) \right)}{1 + \varepsilon \vert u \vert ^{q}} \mu (dt,du).
$$
Let $\varepsilon$ tend to 0. The proof is complete.
\end{proof}

%\medskip
%\medskip

%E(\eta) is non-empty
For any $x \in \T^d$, denote by $\Gamma_{m}^* (x)$ the set of curves associated with an optimal control $\mu^*$, i.e.,
$$
\Gamma_{m}^* (x) := \left\{ \gamma^* \in \Gamma_T^R \, \middle\vert \,\, \gamma^* = \gamma^* \left( \cdot ; x , \mu^*, m \right) , J^m \left( \gamma^*, \mu^* \right) = \inf_{\mu \in \PP_U^R, \gamma = \gamma \left( \cdot ; x , \mu, m \right) } J^m \left( \gamma, \mu \right) \right\}.
$$
Define the map $F_m : \T^d \rightarrow \mathcal{B} \left( \Gamma_T^R \right)$ by
$$
F_m(x) := \Gamma_{m}^* (x), \quad \forall x \in \T^d.
$$

%\medskip

\begin{lemma} \label{F_closed_1}
	Let $m \in \M_r$. Let the sequence $\left\{ x_i \right\}_{i=1}^{\infty}$ converge to $x$ in $\T^d$. If there exists $\gamma_{x_i} \in F_m \left( x_i \right) $ for each $i \in \N$ such that the sequence $\left\{ \gamma_{x_i} \right\}_{i=1}^{\infty}$ converges to some $\gamma$ in $\Gamma_T^R$, then $\gamma \in F_m (x)$.
\end{lemma}
\begin{proof}
Denote by $\mu_{x_i}$ the measure such that $\gamma_{x_i} = \gamma_{x_i} \left( \cdot; x_i, \mu_{x_i},m \right)$ for any $i \in \N$.
Due to the compactness of $\PP_U^R$, there exists a measure $\bar{\mu}$ such that $\tilde{d_1} \left(\mu_{x_i}, \bar{\mu} \right)$ converges to 0 as $ i \rightarrow \infty$.
Let $\tilde{\gamma} = \tilde{\gamma} \left( \cdot; x, \tilde{\mu},m \right)$ and $\tilde{\gamma}_{x_i} = \tilde{\gamma}_{x_i} \left( \cdot; x_i, \tilde{\mu},m \right)$ for any $\tilde{\mu} \in \PP_U^R$. Since $x_i$ converges to $x$ in $\T^d$, we have $\tilde{\gamma}_{x_i}$ converges to $\tilde{\gamma}$ in ${\Gamma}_T^R$.
Moreover, since $\gamma_{x_i} \in \Gamma_{m}^* (x_i)$, we obtain that
\begin{equation} \label{medium---_1}
	J^m \left( \gamma_{x_i}, \mu_{x_i} \right) \leq J^m \left( \tilde{\gamma}_{x_i}, \tilde{\mu} \right).   
\end{equation}
Let $i$ tend to infinity.
By Lemma \ref{Jm_cont_1}, we have $J^m \left( \gamma, \bar{\mu} \right) \leq J^m \left( \tilde{\gamma}, \tilde{\mu} \right)$. 

The only thing left to prove is that $\gamma=\gamma \left( \cdot; x, \bar{\mu},m \right)$, i.e.,
$$
\gamma(t) = x + \int_0^t \int_{\R^n} f \left( \gamma(s), u ,m(s) \right) \bar{\mu} \left( ds,du \right).
$$
Since $\gamma_{x_i}(t) = x_i + \int_0^t \int_{\R^n} f \left( \gamma_{x_i}(s), u , m(s) \right) \mu_{x_i} \left( ds,du \right)$, the proof is complete by (F\ref{f1_1}), (F\ref{f2_1}) and Proposition \ref{D-L A.1}.
\end{proof}

%\medskip

\begin{proposition} \label{nonemp_pre_1}
	For each $m\in \M_r$,
	\begin{enumerate}[(1)]
		\item the set $\Gamma_{m}^* (x)$ is non-empty for any $x\in \T^d$.
		\item the map $F_m$ has a closed graph.
	\end{enumerate}
\end{proposition}
\begin{proof}
The non-emptiness of $\Gamma_{m}^* (x)$ can be easily proved by the convexity of $J^m \left( \gamma, \mu \right)$ with respect to $\mu$. See \cite{MR2041617} for examples.
Point (2) is a direct consequence of Lemma \ref{F_closed_1}.
\end{proof}

%\medskip

\begin{proposition}\label{nonempty and convex_1}
	For any $\eta \in \PP_0 \left( \Gamma_T^R \right)$, the set $E\left( \eta \right)$ is non-empty and convex.
\end{proposition}
\begin{proof}
Let $m(t) = e_t \sharp \eta$ for any $t \in [0,T]$. By Proposition \ref{nonemp_pre_1} and \cite[Proposition 9.5]{MR3361899}, the map $F_m$ is measurable. Thus, there exists a measurable selection $\tilde{\gamma}_x \in \Gamma_{m}^* (x)$ by \cite[Chapter3, Theorem 5.3]{clarke2008nonsmooth}. Define the measure $\tilde{\eta} \in \PP_0 \left( \Gamma_T^R \right)$ by
$$
\tilde{\eta} (A) := \int_{\T^d} \delta_{\left\{ \tilde{\gamma}_x\right\}} (A) m_0 (dx), \quad \forall A \in \mathcal{B} \left( \Gamma_T^R \right).
$$

We claim that $\tilde{\eta} \in E \left( \eta \right)$.
Define the map $h: \left\{ \tilde{\gamma}_x\, \middle\vert \,\, x \in \T^d \right\} \rightarrow \Gamma_T^R \times \PP_{U}^{R}$ by
$$
\tilde{\gamma}_x \mapsto \left( \tilde{\gamma}_x , \mu^x \right), \text{where}\,\, \dot{\tilde{\gamma}}_x (t) = \int_{\R^n} f \left( \tilde{\gamma}_x(t), u, m(t) \right) \mu^x_t(du), \quad \forall t \in [0,T].
$$
It is clear that $\tilde{\eta} = \pi_1 \sharp \left( h \sharp \tilde{\eta} \right)$.
%More precisely, for any function $g\in \CC_b \left( \Gamma_T^R \right)$, we have
%\begin{align*}
%	\int_{\Gamma_T^R} g \left( \gamma \right) \pi_1 \sharp \left( h \sharp \tilde{\eta} \right) \left( d \gamma \right)
%	& = \int_{\Gamma_T^R} g \left( \gamma \right) \mathbf{1}_{\left\{ \tilde{\gamma}_x \, \middle\vert \,\, x \in \T^d \right\} } (\gamma) \tilde{\eta} \left( d \gamma \right) \\
%	& = \int_{\T^d} \int_{\Gamma_T^R} \mathbf{1}_{\left\{ \tilde{\gamma}_x \, \middle\vert \,\, x \in \T^d \right\} }(\gamma) \, g \left( \gamma \right) \delta_{\left\{ \tilde{\gamma}_x\right\}}  \left(d \gamma \right) m_0 \left(dx \right) \\
%	& = \int_{\T^d} \int_{\Gamma_T^R} g \left( \gamma \right) \delta_{\left\{ \tilde{\gamma}_x\right\} }  \left(d \gamma \right) m_0 \left(dx \right) \\
%	& = \int_{\Gamma_T^R} g \left( \gamma \right) \tilde{\eta}  \left( d \gamma \right).
%\end{align*}
Hence, we only need to verify that $h \sharp \tilde{\eta} \in R^* (m)$.
Indeed, we have $h \sharp \tilde{\eta} \in \PP \left( \Gamma_T^R \times \PP_U^R \right)$ and $e_0 \sharp \tilde{\eta} = m_0$ by definition.
For any pair $\left( \gamma, \mu \right) \in \supp \left( h \sharp \tilde{\eta} \right)$, the pair satisfies $\dot{\gamma} = \gamma \left( \cdot ; \gamma(0) , \mu,m \right)$. By Proposition \ref{properties_PR_1}(4), we obtain that $h \sharp \tilde{\eta} \in \PP_R(m)$.

Regarding the minimization property of $h \sharp \tilde{\eta}$, we first define $\Gamma_T^R (m) \subset \Gamma_T^R$ by
$$
\Gamma_T^R (m) := \left\{ \gamma \in \Gamma^R_T \, \middle\vert \,\, \,\exists (x,\mu) \in \T^d \times \PP_U^R, s.t. \,\, \gamma = \gamma \left( \cdot; x,\mu, m \right) \right\}.
$$
%{\color{red} and $\Gamma^R_T(m,x) \subset \Gamma_T^R (m)$ for any $x \in \T^d$ by
%$$
%\Gamma^R_T(m,x) := \left\{ \gamma \in \Gamma^R_T \, \middle\vert \,\, \,\exists \mu \in \PP_U^R, s.t. \,\, \gamma = \gamma \left( \cdot; x,\mu, m \right) \right\}.
%$$}
By Proposition \ref{properties_PR_1}(4), for any $\tilde{P} \in \PP_R(m)$, the support of $\tilde{P}$ satisfies that $\supp(\tilde{P}) \subset \Gamma_T^R (m) \times \PP^R_U$.
%{\color{red} Note that for any $m \in \MMM_r$, $\Gamma^R_T(m) = \bigcup_{x \in \T^d} \Gamma^R_T (m,x)$ and $\Gamma^R_T (m,x_1) \cap \Gamma^R_T (m,x_2) = \emptyset$ for any $x_1 \neq x_2 \in \T^d$.}
Thus, we obtain that
\begin{align*}
	J \left(m, h \sharp \tilde{\eta} \right) 
	& = \int_{\Gamma_T^R \times \PP_{U}^{R}} J^m \left( \gamma, \mu \right) h \sharp \tilde{\eta} \left( d\gamma, d\mu \right) \\
%	& = \int_{\T^d} \int_{\Gamma_T^R} J^m \left( h \left(\gamma \right) \right) \delta_{\left\{ \tilde{\gamma}_x\right\}}  \left(d \gamma \right) m_0 \left(dx \right) \\ 
	& = \int_{\T^d} J^m \left( \tilde{\gamma}_x , \mu^x \right) m_0(dx) \\
	& = \int_{\T^d} J^m \left( \tilde{\gamma}_x , \mu^x \right) e_0 \sharp \left( \pi_1 \sharp \tilde{P} \right)(dx) \\
%	& = \int_{\Gamma_T^R \times \PP_{U}^{R}} J^m \left( \tilde{\gamma}_{\gamma(0)} , \mu^{\gamma(0)} \right) \tilde{P} \left(d\gamma, d \mu \right) \\
	& = \int_{\Gamma_T^R(m) \times \PP_{U}^{R}} J^m \left( \tilde{\gamma}_{\gamma(0)} , \mu^{\gamma(0)} \right) \tilde{P} \left(d\gamma, d \mu \right) \\
%	&{\color{red} = \Sigma_{x \in \T^d} \int_{\Gamma_T^R(m,x) \times \PP_{U}^{R}} J^m \left( \tilde{\gamma}_{\gamma(0)} , \mu^{\gamma(0)} \right) \tilde{P} \left(d\gamma, d \mu \right) }\\
	%&{\color{red} = \Sigma_{x \in \T^d} \int_{\Gamma_T^R(m,x) \times \PP_{U}^{R}} J^m \left( \tilde{\gamma}_{x} , \mu^{x} \right) \tilde{P} \left(d\gamma, d \mu \right)} \\
	%& {\color{red}\leq \Sigma_{x \in \T^d} \int_{\Gamma_T^R(m,x) \times \PP_{U}^{R}}J^m \left( \gamma , \mu \right) \tilde{P} \left(d\gamma, d \mu \right) }\\
	& \leq \int_{\Gamma_T^R(m) \times \PP_{U}^{R}} J^m \left( \gamma , \mu \right) \tilde{P} \left(d\gamma, d \mu \right) \\
	& = \int_{\Gamma_T^R \times \PP_{U}^{R}} J^m \left( \gamma , \mu \right) \tilde{P} \left(d\gamma, d \mu \right) = J \left( m, \tilde{P} \right).
\end{align*}

Therefore, the measure $\tilde{\eta} \in E \left( \eta \right)$, which implies the non-emptiness of $E \left( \eta \right)$.

The convexity can be easily proved by definition.
\end{proof}

%\medskip

%E has a closed graph
\begin{proposition}\label{main_1}
	Let $\left\{\eta_i\right\}_{i=1}^\infty \subset \PP_0 \left( \Gamma_T^R \right)$ and $\eta \in \PP_0 \left( \Gamma_T^R \right)$ such that $\eta_{i}$ converges to $\eta$ narrowly. If there exists $\hat{\eta}_i \in E\left(\eta_{i} \right)$ for any $i \in \N$ such that $\hat{\eta}_{i}$ converges to some $\hat{\eta}$ narrowly, then $\hat{\eta} \in E(\eta)$.
\end{proposition}
\begin{proof}
For any $t \in [0,T]$, define $m^i(t)= e_t \sharp \eta_i$ and $m(t)= e_t \sharp \eta$.
%We first claim that the function $J^{m^i} \left( \gamma, \mu \right):\Gamma_T^R \times \PP_{U}^{R} \rightarrow \R$ uniformly converges to the function $J^m(\gamma, \mu):\Gamma_T^R \times \PP_{U}^{R} \rightarrow \R$ as $i \rightarrow \infty$.
By (L\ref{L2_1}),
\begin{align*}
	\left \vert J^{m^i} \left( \gamma, \mu \right) - J^m(\gamma, \mu) \right\vert 
	&\leq \int_0^T \int_{\R^n} \left \vert L \left( \gamma(t), u, m^i(t) \right) - L \left( \gamma(t), u, m(t) \right) \right \vert \mu_t (d u) d t \\
	&  \leq \int_0^T \int_{\R^n} \omega \left( d_1 \left( m^i(t), m(t) \right) \right) \mu_t(d u) d t 
    = \int_0^T \omega \left( d_1 \left( m^i(t), m(t) \right) \right) dt.
\end{align*}
Indeed, since $\eta_i$ converges to $\eta$ narrowly and $\Gamma _T^R$ is compact, we have $d_1 \left( \eta_i , \eta \right) \rightarrow 0$. Furthermore, we obtain that $d_1 \left( m^i(t) , m(t) \right) \leq d_1 \left( \eta_i , \eta \right)$ for any $t \in [0,T]$ by definition.
%\begin{align*}
%	d_1 \left( m^i(t) , m(t) \right) 
%	& = \sup \left\{ \int_{\T^d} g(x) m^i(t) (dx) - \int_{\T^d} g(x) m(t) (dx)\, \middle\vert \,\, g(x) \, \text{is $1$-Lip function w.r.t}\,\, x\right\} \\
%	& \leq \sup \left\{ \int_{\Gamma _T^R} g\left(\gamma(t)\right) \eta_i (d \gamma) - \int_{\Gamma _T^R} g\left(\gamma(t)\right) \eta (d \gamma)\, \middle\vert \,\, g\left( \gamma(t) \right) \, \text{is $1$-Lip function w.r.t}\,\, \gamma \right\} \\
%	& = d_1 \left( \eta_i , \eta \right).
%\end{align*}
We obtain that
\begin{equation} \label{J^m^i converges to J^m}
\left \vert J^{m^i} \left( \gamma, \mu \right) - J^m(\gamma, \mu) \right\vert \leq T \omega \left( d_1 \left( \eta_i , \eta \right) \right).
\end{equation}
%Thus, the claim is established.
Thus, the function $J^{m^i} \left( \gamma, \mu \right)$ uniformly converges to the function $J^m(\gamma, \mu)$ as $i \rightarrow \infty$.
By (\ref{J^m^i converges to J^m}), the function $\tilde{P} \mapsto J \left( m^i, \tilde{P} \right)$ also uniformly converges to the function $\tilde{P} \mapsto J \left( m, \tilde{P} \right)$ as $i \rightarrow \infty$, i.e.,
\begin{equation} \label{4.1}
	\lim_{i \rightarrow \infty} \sup_{\tilde{P} \in \PP \left(\Gamma _T^R \times \PP_U^R \right)} \left\vert J \left( m^i, \tilde{P} \right) - J \left( m, \tilde{P} \right) \right\vert =0.
\end{equation}

For any $i \in \N$, since $\hat{\eta}_i \in E\left(\eta_{i} \right)$, there exists $P_i \in \PP_R \left(m^i \right)$ such that $\pi_1 \sharp P_i = \hat{\eta}_i$.
Since $\PP \left( \Gamma_T^R \times \PP_{U}^{R} \right)$ is compact with respect to the $d_1$-topology, there exists $P \in \PP \left( \Gamma_T^R \times \PP_{U}^{R} \right)$ such that $d_1 \left( P_i,P \right) \rightarrow 0$ as $i \rightarrow \infty$.
%It is obviously that $\pi_1 \sharp P = \hat{\eta}$.
%More precisely, for any function $g_1 \in \CC_b \left( \Gamma_T^R \right)$, we have
%\begin{align*}
%	\int_{\Gamma_T^R} g_1 \left(\gamma \right) \pi_1 \sharp P \left( d \gamma \right)
%	& = \int_{\Gamma_T^R \times \PP_{U}^{R}} g_1 \left(\gamma \right)P \left( d \gamma, d \mu \right) \\
%	& = \lim_{i \rightarrow \infty} \int_{\Gamma_T^R \times \PP_{U}^{R}} g_1 \left(\gamma \right)P_i \left( d \gamma, d \mu \right) \\
%	& = \lim_{i \rightarrow \infty} \int_{\Gamma_T^R} g_1 \left(\gamma \right) \hat{\eta_i} \left( d \gamma \right) \\
%	& = \int_{\Gamma_T^R} g_1 \left(\gamma \right) \hat{\eta} \left( d \gamma \right).
%\end{align*}
It is obviously that $\pi_1 \sharp P = \hat{\eta}$ and $e_0 \sharp \pi_1 \sharp P = e_0 \sharp \hat{\eta} = m_0$.
%Moreover, we have $e_0 \sharp \pi_1 \sharp P = e_0 \sharp \hat{\eta} = m_0$ by definition. For any function $g_2 \in \CC_b \left( \T^d \right)$, we obtain that
%\begin{align*}
%	\int_{\T^d} g_2(x) e_0 \sharp \hat{\eta} (dx)
%	& = \int_{\Gamma _T^R} g_2 \left( \gamma(0) \right) \hat{\eta} \left( d \gamma \right) \\
%	& = \lim_{i \rightarrow \infty} \int_{\Gamma _T^R} g_2 \left( \gamma(0) \right) \hat{\eta}_i \left( d \gamma \right) \\
%	& = \lim_{i \rightarrow \infty} \int_{\T^d} g_2(x) e_0 \sharp \hat{\eta}_i (dx) \\
%	& = \lim_{i \rightarrow \infty} \int_{\T^d} g_2(x) e_0 \sharp \pi_1 \sharp P_i (dx) = \int_{\T^d} g_2(x) m_0(dx).
%\end{align*}
For any $(t,v) \in [0,T] \times \R^d$ and any open set $N \subset \Gamma_T^R \times \PP_{U}^{R}$, consider
\begin{align*}
	& \left\vert \int_{N} \left \langle v, \gamma(t) - \gamma(0) - \int_0^t \int_{\R^n} f \left( \gamma(s), u,m^i(s) \right) \mu \left( ds,du \right) \right\rangle P_i \left( d\gamma, d \mu \right) \right. \\
	& -\left. \int_{N} \left\langle v, \gamma(t) - \gamma(0) - \int_0^t \int_{\R^n} f \left( \gamma(s), u,m(s) \right) \mu \left( ds,du \right) \right\rangle P \left( d\gamma, d \mu \right) \right\vert  \leq B_1 +B_2,
\end{align*}
where
\begin{align*}
	B_1 := & \left\vert \int_{N} \left \langle v, \gamma(t) - \gamma(0) - \int_0^t \int_{\R^n} f \left( \gamma(s), u,m^i(s) \right) \mu \left( ds,du \right) \right \rangle P_i \left( d\gamma, d \mu \right) \right. \\
	-& \left. \int_{N} \left \langle v, \gamma(t) - \gamma(0) - \int_0^t \int_{\R^n} f \left( \gamma(s), u,m(s) \right) \mu \left( ds,du \right) \right \rangle P_i \left( d\gamma, d \mu \right) \right\vert,
\end{align*}
\begin{align*}
B_2 := & \left\vert \int_{N} \left \langle v, \gamma(t) - \gamma(0) - \int_0^t \int_{\R^n} f \left( \gamma(s), u,m(s) \right) \mu \left( ds,du \right) \right \rangle P_i \left( d\gamma, d \mu \right) \right. \\
	- & \left. \int_{N} \left \langle v, \gamma(t) - \gamma(0) - \int_0^t \int_{\R^n} f \left( \gamma(s), u,m(s) \right) \mu \left( ds,du \right) \right \rangle P \left( d\gamma, d \mu \right) \right\vert.
\end{align*}
Similar to the proof of Proposition \ref{properties_PR_1}(1), we have $\lim_{i \rightarrow \infty} B_2 =0$.
At the same time, since
\begin{align*}
	B_1 
	& =\left\vert \int_{N}  \left \langle v , \int_0^t \int_{\R^n} f \left( \gamma(s) , u, m(s) \right) - f \left( \gamma(s) , u, m^i(s) \right) \mu (ds,du) \right \rangle  P_i \left( d\gamma, d \mu \right) \right\vert \\
	& \leq  \int_{N}   \left\vert v \right\vert \int_0^t \Lip_{\nu}(f) d_1 \left( \eta_i , \eta \right) ds \, P_i \left( d\gamma, d \mu \right) \leq \left\vert v \right\vert \Lip_{\nu}(f) T d_1 \left( \eta_i , \eta \right),
\end{align*}
we have $\lim_{i \rightarrow \infty} B_1 =0$.
Since $P_i \in \PP_R \left( m^i \right)$, we obtain that $P \in \PP_R (m)$.

Consider
\begin{equation*}
	J\left( m^i, P_i \right) - J \left( m, P \right) =
	\left( J\left( m^i,P_i \right) - J \left( m, P_i \right) \right)  + \left( J \left( m, P_i \right) - J \left( m, P \right) \right) \triangleq B_3+B_4.
\end{equation*}
By (\ref{4.1}), $\lim_{i \rightarrow \infty}B_3 = 0$.
Since the function $\left( \gamma, \mu \right) \mapsto J^m \left( \gamma, \mu \right)$ is bounded and lower semi-continuous by Lemma \ref{Jm_cont_1}, it follows from Lemma \ref{lemma of s representation theorem} that $\liminf_{i \rightarrow \infty} B_4 \geq 0$.

In conclusion, for any $\hat{P} \in \PP_R(m)$, we have 
$$
J \left( m,P \right) \leq \liminf_{i \rightarrow \infty} J \left( m^i, P_i \right) \leq \liminf_{i \rightarrow \infty} J \left( m^i, \hat{P} \right) = J \left( m, \hat{P} \right).
$$
Due to the arbitrariness of $\hat{P} \in \PP _R(m)$, we complete the proof.
\end{proof}

%\medskip

%E(\eta) is compact
\begin{corollary} \label{compact_1}
	For any $\eta \in \PP_0 \left( \Gamma_T^R \right)$, the set $E \left( \eta \right)$ is compact.
\end{corollary}
\begin{proof}
By Proposition \ref{main_1}, the set-valued map $E$ has a closed graph, which ensures that $E \left( \eta \right)$ is a closed set. Furthermore, $E(\eta)$ is contained within the compact set $\PP_0 \left( \Gamma_T^R \right)$.
Consequently, $E \left( \eta \right)$ is also compact.
\end{proof}

%\medskip

%proof of main result 1
\begin{proof}[Proof of Theorem \ref{existence_1}]
	Combining Proposition \ref{nonempty and convex_1} and Corollary \ref{compact_1}, the set $E \left( \eta \right)$ is non-empty, convex and compact for any $\eta \in \PP_0 \left( \Gamma_T^R \right)$. Furthermore, $\PP_0 \left( \Gamma_T^R \right)$ is non-empty and compact. By Proposition \ref{main_1}, the set-valued map $E$ has a closed graph, which indicates the continuity of the map. Consequently, there exists a measure $\bar{\eta} \in \PP_0 \left( \Gamma_T^R \right)$, such that $\bar{\eta} \in E \left( \bar{\eta} \right)$ by Kakutani's theorem.
	Thus, there exists $P \in \PP \left(\Gamma _T^R \times \PP_U^R \right)$ such that $\bar{\eta} = \pi_1 \sharp P$ and $P \in R^*\left( \left( e_t \sharp \bar{\eta} \right)_{t \in [0,T]} \right)$. Therefore, the probability measure $P$ belongs to the set $R^* \left( \left( e_t \sharp \pi_1 \sharp P \right)_{t \in [0,T]} \right)$, indicating that $P$ is a relaxed MFC equilibrium.
\end{proof}

\section{Existence of strict relaxed MFC equilibria}\label{strictMFG}
%well-defined of strict relaxed MFG equilibria
\begin{lemma} \label{pre1}
	Fix $m \in \M_r$. For any $x \in \T^d$ and any $u_x: [0,T] \rightarrow \R^n$ such that $\int_0^T \left\vert u_x(t) \right\vert^q dt \leq R$, denote by $\gamma_x^{u_x} = \gamma_x^{u_x} \left( \cdot ; x, \LL_{[0,T]} \otimes \delta_{u_x(t)},m \right)$ the solution of the state equation (\ref{traj_1}). % when the measure $\LL_{[0,T]} \otimes \delta_{u(t)} \in \PP^R_U$.
	Define a probability measure $\eta^u \in \PP_0 \left( \Gamma_T^R \right)$ by
	$$
	\eta^u(A) := \int_{\T^d} \delta_{ \left\{\gamma_x^{u_x}\right\} }(A) m_0(dx), \quad \forall A \in \mathcal{B} \left( \Gamma_T^R \right).
	$$
	For any $\eta \in \PP_0 \left( \Gamma_T^R \right)$, define $P_\eta := \eta \otimes \delta_{\left\{\LL_{[0,T]} \otimes \delta_{u_{\gamma(0)}(t)} \right\}}$. Then, $P_\eta \in \PP_R(m)$ if and only if $\eta = \eta^u$.
	Denote $P^u := \eta^u \otimes \delta_{\left\{\LL_{[0,T]} \otimes \delta_{u_{\gamma(0)}(t)} \right\}} \in \PP_R(m)$.
\end{lemma}
\begin{proof}
Obviously, $P^u \in \PP_R(m)$ by definition.

Suppose $\eta \in \PP_0 \left( \Gamma_T^R \right)$ and $P_\eta \in \PP_R(m)$. For any $\gamma \in \supp(\eta)$, by Proposition \ref{properties_PR_1}(4), there exists $x \in \T^d$ such that $\gamma(0)=x$ and $\gamma = \gamma_x^{u_x}$. By $e_0 \sharp \eta = e_0 \sharp \eta^u = m_0$, for any $g \in \CC_b \left( \Gamma_T^R \right)$,
$$
\int_{\Gamma^R_T} g(\gamma) \eta(d\gamma) = \int_{\Gamma^R_T} g(\gamma_{\gamma(0)}^{u_{\gamma(0)}}) \eta(d\gamma) = \int_{\T^d} g (\gamma_x^{u_x}) m_0(dx) =\int_{\Gamma^R_T} g(\gamma) \eta^u(d\gamma).
$$
Thus, $\eta = \eta^u$.

%	It is clear that $P^u \in \PP \left( \Gamma_T^R \times \PP_U^R\right)$ and $e_0 \sharp \left( \pi_1 \sharp P^u \right) = m_0$. Moreover, for any $(t,v) \in [0,T] \times \R^d$ and any open set $N \subset \Gamma_T^R \times \PP_{U}^{R}$, we can deduce that
%	\begin{align*}
%		& \int_{N} \left \langle v, \gamma(t) - \gamma(0) - \int_0^t \int_{\R^n} f \left( \gamma(s), u,m(s) \right) \mu \left( ds,du \right) \right \rangle P^u \left( d\gamma, d \mu \right) \\
%		= & \int_{\pi_1(N)} \left \langle v, \gamma(t) - \gamma(0) - \int_0^t f \left( \gamma(s), u(s),m(s) \right) ds \right \rangle \eta^u \left( d \gamma \right) \\
%		= & \int_{\T^d} \int_{\pi_1(N)} \left \langle v, \gamma(t) - \gamma(0) - \int_0^t f \left( \gamma(s), u(s),m(s) \right) ds \right \rangle \delta_{\left\{ \gamma_x^u \right\}}\left( d \gamma \right) m_0(dx) \\
%		= & \int_{\left\{x\, \middle\vert \, \gamma_x^u \in \pi_1(N) \right\}} \left \langle v, \gamma_x^u (t) - \gamma_x^u (0) - \int_0^t f \left( \gamma_x^u (s), u(s),m(s) \right) ds \right\rangle m_0(dx) \\
%		= & \int_{\left\{x\, \middle\vert \, \gamma_x^u \in \pi_1(N) \right\}} \left \langle v, \gamma_x^u (t) - \gamma_x^u (0) - \int_0^t \dot{\gamma}_x^u (s) ds \right \rangle m_0(dx) = 0.
%	\end{align*}
%	Therefore, the probability measure $P^u$ belongs to the set $\PP_R(m)$.
\end{proof}

%\medskip

\begin{proposition} \label{3.7_1_1}
	Fix $m \in \M_r$ and $P \in R^* \left( m \right).$ There exists a map $\hat{q}: \left[ 0 , T \right] \times \T^d \rightarrow \PP \left( \R^n \right)$ such that $\hat{P} := \pi_1 \sharp P \otimes \delta_{ \left\{ \LL_{[0,T]} \otimes \hat{q}_{t,\gamma(t)}\right\} } \in R^*(m)$,
	where we denote $\hat{q}_{t,x} = \hat{q} \left( t,x \right)$ for convenience.
\end{proposition}
\begin{proof}
Define a measure $\eta$ on $[0,T] \times \T^d \times \R^n$ by
$$
\eta\left( C \right) := \frac{1}{T} \int_{\Gamma^R_T \times \PP^R_U} \int_0^T \int_{\R^n} \mathbf{1}_C \left( t, \gamma(t), u \right) \mu \left( dt, du \right) P \left( d \gamma, d \mu \right), \quad \forall C \in \mathcal{B} \left( [0,T] \times \T^d \times \R^n \right).
$$
Define $\hat{\eta}_t := e_t \sharp \left( \pi_1 \sharp P \right)$, which is a probability measure on $\T^d$. 
Define $\eta_{1,2} := \frac{1}{T} \LL_{[0,T]} \otimes \hat{\eta}_t (dx)$, which is a probability measure on $[0,T] \times \T^d $. By Theorem \ref{disintegration_thm}, construct the map $\hat{q}$ by
$$
\eta \left( dt,dx,du \right) = \eta_{1,2} (dt,dx) \otimes \hat{q}_{t,x} (du).
$$ 
Define the measure $\hat{P}$ by $\hat{P} := \pi_1 \sharp P \otimes \delta_{ \left\{ \LL_{[0,T]} \otimes \hat{q}_{t,\gamma(t)}\right\} } \in \PP \left( \Gamma^R_T \times \PP^R_U \right)$.
Then, we obtain that $\hat{P} \in \PP_R(m)$ by Appendix \ref{Completion of the Proof of Proposition 3.7_1_1}.
Furthermore, 
\begin{equation*}
\begin{aligned}
		J(m, \hat{P}) 
		& = \int_{\Gamma^R_T } \int_0^T \int_{\R^n} L \left( \gamma(t), u, m(t) \right) \hat{q}_{t,\gamma(t)} \left(d u \right)\, dt \,\pi_1 \sharp P \left( d \gamma \right) \\
		& = \int_0^T \int_{\T^d} \int_{\R^n} L \left( x,u,m(t) \right) \hat{q}_{t,x} (du)\,e_t \sharp \left( \pi_1 \sharp P \right) (dx)\, dt\\
%		&=T \int_0^T \int_{\T^d}  \int_{\R^n} L \left( x,u,m(t) \right) \hat{q}_{t,x} (du)\, \eta_{1,2} (dt,dx)\\
		&= T\int_0^T \int_{\T^d} \int_{\R^n} L \left( x,u,m(t) \right) \eta \left( dt,dx,du \right)\\
		&=  \int_{\Gamma^R_T \times \PP^R_U} \int_0^T \int_{\R^n} L \left( \gamma(t), u, m(t) \right) \mu (dt,du) P \left( d \gamma, d \mu \right) = J(m,P).
	\end{aligned}
\end{equation*}
Therefore, we have $\hat{P} \in R^*(m)$.

\end{proof}

%\medskip

%proof of main result 2
\begin{proof}[Proof of Theorem \ref{relation_with_convex_case}]
(1) Define the map $\hat{q} : [0,T] \times \T^d \rightarrow \PP \left( \R^n \right)$ by Proposition \ref{3.7_1_1}. By the convexity condition, for any $(t, x) \in [0,T] \times \T^d$, we have
$$
\int_{\R^n} \left( L \left( x,u,m(t) \right),\vert u \vert^q, f\left( x,u,m(t) \right) \right) \hat{q}_{t,x} (du) \in \mathcal{L} \left(t,x \right).
$$
By \cite[Theorem A.9]{28_A9}, there exist negative measurable functions $z_1: [0,T] \times \T^d \rightarrow \R^{-}, z_2: [0,T] \times \T^d \rightarrow \R^{-}$ and a measurable map $\hat{\alpha}: [0,T] \times \T^d \rightarrow \R^n$ such that 
$$
L \left( x, \hat{\alpha}(t,x),m(t) \right) = \int_{\R^n} L \left( x,u,m(t) \right) \hat{q} _{t,x} (du) +z_1(t, x),
$$
$$
\left\vert \hat{\alpha}(t,x) \right\vert^q = \int_{\R^n} \vert u \vert^q \hat{q} _{t,x} (du) +z_2(t, x),
$$
$$
f \left( x, \hat{\alpha}(t,x),m(t) \right) = \int_{\R^n} f \left( x,u,m(t) \right) \hat{q} _{t,x} (du).
$$
Define the probability measure $P_0 : = \pi_1 \sharp P \otimes\delta_{ \left\{\LL_{[0,T]} \otimes \delta_{\hat{\alpha} \left( t,\gamma(t) \right)}\right\} } \in \PP \left( \Gamma_T^R \times \PP_{U}^{R} \right)$. The measure $P_0$ is well-defined since for any $\gamma \in \supp \left( \pi_1 \sharp P \right)$, we have
\begin{align*}
	\int_0^T \left\vert \hat{\alpha}(t,\gamma(t)) \right\vert^q dt
	& = \int_0^T \int_{\R^n} \vert u \vert^q \hat{q} _{t,\gamma(t)} (du) dt + \int_0^T z_2(t, \gamma(t))dt \\
	& \leq \int_0^T \int_{\R^n} \vert u \vert^q \hat{q} _{t,\gamma(t)} (du) dt \leq R.
\end{align*}

It is clear that $e_0 \sharp \left( \pi_1 \sharp P_0 \right) = e_0 \sharp \left( \pi_1 \sharp P \right) = m_0$ and $P_0 \in \PP_R(m)$.
%For any $(t,v) \in [0,T] \times \R^d$ and any open set $N \subset \Gamma_T^R \times \PP_{U}^{R}$, we have
%\begin{align*}
%	&  \int_{N} \left \langle v, \gamma(t) - \gamma(0) - \int_0^t \int_{\R^n} f \left( \gamma(s), u,m(s) \right) \mu \left( ds,du \right) \right \rangle P_0 \left( d\gamma, d \mu \right)  \\
%	= & \int_{\pi_1(N)} \left \langle v, \gamma(t) - \gamma(0) - \int_0^t f \left( \gamma(s), \hat{\alpha} \left( s,\gamma(s) \right),m(s) \right) ds \right \rangle \pi_1 \sharp P \left( d \gamma \right) \\
%	= & \int_{\pi_1(N)} \left \langle v, \gamma(t) - \gamma(0) - \int_0^t \int_{\R^n} f \left( \gamma(s), u ,m(s) \right) \hat{q}_{s, \gamma(s)}(du) ds \right\rangle \pi_1 \sharp P \left( d \gamma \right) = 0.
%\end{align*}
%Therefore, we conclude that $P_0 \in \PP_R(m)$.
Besides,
\begin{align*}
	J \left( m, P_0 \right) 
	&= \int_{\Gamma^R_T } \int_0^T \int_{\R^n} L \left( \gamma(t), u, m(t) \right) \delta_{\hat{\alpha} \left( t, \gamma(t) \right)} \left(du \right)\, dt \, \pi_1 \sharp P \left( d \gamma \right)\\
	&= \int_{\Gamma^R_T } \int_0^T L \left( \gamma(t), \hat{\alpha} \left( t, \gamma(t) \right), m(t)  \right)\, dt \, \pi_1 \sharp P \left( d \gamma \right)\\
	&= \int_{\Gamma^R_T } \int_0^T \int_{\R^n} L \left( \gamma(t), u, m(t) \right) \hat{q} _{t,\gamma(t)} (du) \, dt \, \pi_1 \sharp P \left( d \gamma \right) + \int_{\Gamma^R_T} \int_0^T z_1 \left( t, \gamma (t) \right) dt\, \pi_1 \sharp P \left( d \gamma \right)\\
	& \leq J \left( m, \hat{P} \right) = J \left( m, P \right).
\end{align*}
In conclusion, $P_0 \in R^*(m)$.

Since $P_0 \in \PP_R (m)$, for any $\gamma \in \supp \left( \pi_1 \sharp P \right)$, it satisfies
\begin{equation} \label{sol for alpha}
\dot{\gamma} (t) = f \left( \gamma(t), \hat{\alpha} \left(t, \gamma(t) \right), m(t) \right).
\end{equation}
Denote by $\gamma^{\hat{\alpha}}_x$ the solution of \eqref{sol for alpha} with $\gamma^{\hat{\alpha}}_x(0) = x$. Define the maps $u^*_x(t):= \hat{\alpha} \left(t, \gamma^{\hat{\alpha}}_x(t) \right)$ for any $x \in \T^d$.
Then, by $e_0 \sharp \pi_1 \sharp P = m_0$, it is clear that
$$
J(m,P_0) = \int_{\T^d} \int_0^T L \left( \gamma^{u^*_x}_x (t), u^*_x(t), m(t) \right) dt\, m_0(dx),
$$
where
$$
\dot{\gamma}^{u^*_x}_x (t) = f \left( \gamma^{u^*_x}_x(t), u_x^*(t), m(t) \right), \quad \gamma^{u^*_x}_x(0)=x.
$$
By Lemma \ref{pre1}, for any $x \in \T^d$ and any map $u_x:[0,T] \rightarrow \R^n$ such that $\int_0^T \left \vert u_x(t) \right\vert^q dt \leq R$, we have $P^u = \eta^{u} \otimes \delta_{\left\{\LL_{[0,T]} \otimes \delta_{u_{\gamma(0)}(t)} \right\} } \in \PP_R(m)$. Thus,
$$
\int_{\T^d} \int^T_0 L \left( \gamma^{u_x}_x (t) , u_x(t) ,m(t) \right) d t\, m_0(dx) \geq \int_{\T^d} \int^T_0 L \left( \gamma^{u^*_x}_x (t) , u^*_x(t) ,m(t) \right) d t\, m_0(dx).
$$

%\medskip

(2) By Theorem \ref{existence_1}, there exists a relaxed MFC equilibrium $P \in R^* (m)$, where $m(t)= e_t \sharp \left( \pi_1 \sharp P \right)$ for any $t \in [0,T] $.
As a consequence of the point (1) of this theorem, $P_0 \in R^* \left( m\right)$.
Define $m^0(t) :=  e_t \sharp \left( \pi_1 \sharp P_0 \right)$ for any $t \in [0,T]$.
By the definition of $P_0$, we have $m(t) = m^0(t)$ for any $t \in [0,T]$.
Thus, $P_0 \in R^*\left(m^0 \right)$, which means $P_0$ is also a relaxed MFC equilibrium, and it is strict obviously.
\end{proof}

\appendix
%\section{Appendix} \label{appendix}
\section{Notations} \label{assumtions and definitions}
We list some notations that are used in this paper as follows.
\begin{itemize}
	\item For any set $A$, denote by $\mathcal{B} (A)$ the family of Borel subsets of $A$, by $\PP (A)$ the family of Borel probability measures on $A$, by $\MMM (A)$ the family of Borel measures on $A$. 
	The support of a measure $P \in \MMM (A)$, denoted by $\supp \left( P \right)$, is a closed set defined by
	$$
	\supp \left( P \right) := \left\{ x \in A \, \middle\vert \,\, P (V_x) > 0 \,\, \text{for every open neighborhood} \,\, V_x \,\, \text{of}\,\,  x \right\}.
	$$

	\item Denote by $\R$ the set of real numbers, by $\R^-$ the set of negative real numbers, by $\Z$ the set of integers, by $\N$ the set of positive integers. For any $n \in \N$, denote by $\R^n$ the $n$-dimensional real Euclidean space, by $\langle \cdot, \cdot \rangle$ the Euclidean scalar product in $\R^n$, by $\vert \cdot \vert$ the usual norm in $\R^n$.
	Denote by $\T^n = \R^n / \Z^n$ the standard $n$-dimensional flat torus. Let $\PP \left( \T^d \right)$ be endowed with the narrowly convergence. It is convenient to put a metric, i.e., the Kantorovich-Rubinstein distance $d_1$ on $\PP \left( \T^d \right)$.

	\item Fix a constant $T>0$. For any absolutely continuous map $\gamma: [0,T] \rightarrow \R^d$, denote by $\Vert \gamma \Vert_{\infty}$ the uniform norm of the map $\gamma$, i.e.,
	$$
	\Vert \gamma \Vert_{\infty} = \sup_{t \in [0,T]} \vert \gamma(t) \vert.
	$$
	Denote by $\mathcal{L}_{[0,T]}$ the Lebesgue measure on $[0,T]$. For any constant $1 \leq \alpha <+ \infty$, denote by $\left\Vert \gamma \right\Vert_{L_\alpha \left( \mathcal{L}_{[0,T]} \right)}$ the $L_\alpha$ norm of the map $\gamma$ with respect to the measure $\mathcal{L}_{[0,T]}$, i.e.,
	$$
	\left\Vert \gamma \right\Vert_{L_\alpha \left( \mathcal{L}_{[0,T]} \right)} = \left(\int_0^T \left\vert \gamma(t) \right\vert^\alpha dt \right)^{\frac{1}{\alpha}}.
	$$
	
	\item Let $\Gamma_T$ be a set of absolutely continuous maps $\gamma: [0,T] \rightarrow \T^d$.  For any $t \in [0,T]$, denote by $e_t$ the evaluation map, i.e.,
	$$
	e_t(\gamma) = \gamma(t), \quad \forall\; \gamma \in \Gamma_T.
	$$
	
	\item The function $\omega: \R^+ \rightarrow \R^+$ is a modulus function if it is a nondecreasing and upper semi-continuous function such that $\displaystyle{\lim_{r \rightarrow 0+}} \omega(r) =0$.

	\item Let A be a set. For any subset $B \subset A$, denote by $\mathbf{1}_B: A \rightarrow\{0,1\}$ the indicator function of $B$, i.e.,
	$$
	\mathbf{1}_B(x) = 
	\begin{cases}
		1 & x \in B, \\ 
		0 & x \notin B .
	\end{cases}
	$$
	Let $x \in A$. Denote by $\delta_{\left\{x\right\}}$ the Dirac mass at point $x$.

	\item For any sets $N_1,N_2$ and any $P \in \PP \left( N_1 \times N_2 \right)$, denote by $\pi_1: N_1 \times N_2 \rightarrow N_1$ the canonical projection of the first variable. Define the push-forward measure $\pi_1 \sharp P \in \PP (N_1)$ by
	$$
	\pi_1 \sharp P \left( B \right) : = P \left( \pi_1^{-1} \left(B \right) \right), \quad \forall B \in \mathcal{B}\left( N_1 \right).
	$$
	Let a probability measure $\eta \in \PP(N_1)$ and a family of probability measures $\left(\nu_x \right)_{x \in N_1} \subset \PP(N_2)$. Denote $P = \eta \otimes \nu_x$ if these measures satisfy
	$$
	\int_{N_1 \times N_2} g(x,y) P(dx,dy) = \int_{N_1} \int_{N_2} g(x,y) \nu_x(dy) \eta(dx), \quad \forall\,\, \text{Borel map}\,\, g: N_1 \times N_2 \rightarrow [0, +\infty].
	$$
	Similarly, let a probability measure $\nu \in \PP(N_2)$ and a family of probability measures $\left(\eta_y \right)_{y \in N_2} \subset \PP(N_1)$. Denote $P = \eta_y \otimes \nu$ if these measures satisfy
	$$
	\int_{N_1 \times N_2} g(x,y) P(dx,dy) = \int_{N_2} \int_{N_1} g(x,y) \eta_y(dx) \nu(dy), \quad \forall\,\, \text{Borel map}\,\, g: N_1 \times N_2 \rightarrow [0,+\infty].
	$$

	\item For metric spaces $\left( [0,t], d_A \right)$ and $\left( B, d_B \right)$ with $t>0$, denote by $AC \left( [0,t], B \right)$ the space of absolutely continuous maps from $[0,t]$ to $B$. 
	For metric spaces $\left( A, d_A \right)$ and $\left( B, d_B \right)$, denote by $\CC \left( A, B \right)$ the space of continuous maps from $A$ to $B$. Denote by $\CC \left( A\right)$ the space of continuous functions on $A$. Denote by $\CC_b \left( A \right)$ the space of bounded and continuous functions on $A$. Denote by $\CC^k \left( A \right)$ the space of $k$-times continuously differentiable functions on $A$.
\end{itemize}

\section{Measure theory} \label{measure_theory}
In this section, we recall some results from measure theory that will be useful in this paper.
Throughout this section, the space $\left( M, \rho \right)$ is a separable metric space.

%\medskip

For a sequence $\left\{ \mu_n \right\}_{n =1}^{\infty} \subset \PP (M)$, $\mu_n$ narrowly converges to some $\mu \in \PP(M)$ as $n \rightarrow \infty$ if 
$$
\lim_{n \rightarrow \infty} \int_M f(x) \mu_n (dx) = \int_M f(x) \mu (dx), \quad \forall f \in \CC_b \left( M \right).
$$

For each $p \in [1, +\infty )$, the Wasserstein space of order $p$ is defined by
$$
\PP_p (M) := \left\{ \mu \in \PP(M)\, \middle \vert \,\, \int_M \rho^p \left( x_0 , x \right)  \mu (d x) < +\infty \right\},
$$
where $x_0 \in M$ is an arbitrary point.
The Monge-Kantorovich distance on $\PP_p(M)$ is defined by
$$
d_p\left(\mu_1, \mu_2 \right) :=\inf _{\eta \in \Pi(\mu_1, \mu_2)} \left( \int_{M^2} \rho^p(x, y)  \eta \left( d x, dy\right)\right)^{1 / p}, \quad \forall \mu_1,\mu_2 \in \PP_p(M),
$$
where $\Pi(\mu_1, \mu_2)$ is the set of Borel probability measures on $M^2$ such that $\eta \left(A \times M\right)=\mu_1(A)$ and $\eta \left(M \times A\right)=\mu_2(A)$ for any Borel set $A \in \mathcal{B}(M)$.

As for the distance $d_1$, which is often called Kantorovich-Rubinstein distance, can be characterized by a useful duality formula (see, for instance, \cite{2013}) as
$$
d_1 (\mu_1, \mu_2) = \sup \left\{ \int_M g(x) \mu_1(d x) - \int_M g(x)  \mu_2(dx) \right\}, \quad \forall \mu_1, \mu_2 \in \PP_1 (M), 
$$
where the supremum is taken over all 1-Lipschitz functions $g: M \rightarrow \R$.

\medskip
\medskip

We now recall the relations between the narrowly convergence and the $d_p$-convergence. See \cite[Theorem 7.1.5]{MR2401600} and \cite[Theorem 7.12]{MR1964483} for examples.
\begin{proposition}
	If a sequence of measures $\left\{ \mu_i \right\}_{i =1}^{\infty} \subset \PP_p (M)$ converges to some $\mu \in \PP_p (M)$ in $d_p$-topology, then $\mu_n$ converges to $\mu$ narrowly.
	Conversely, if $\bigcup_{i=1}^\infty \supp \{\mu_i\}$ is contained in a compact subset of $M$ and $\mu_i$ converges to $\mu$ narrowly, then $\mu_i$ converges to $\mu$ in $d_p$-topology.
\end{proposition}

\begin{proposition} \label{D-L A.1}
	Let $\mu\in \PP_p (M)$ and $\left\{ \mu_i\right\}_{i=1}^\infty \subset \PP_p (M)$. The following statements are equivalent.
	\begin{enumerate}[(1)]
		\item $d_p \left( \mu_i, \mu \right) \rightarrow 0$.
		
		\item $\mu_i$ converges to $\mu$ narrowly. There exists $x_0 \in M$ such that
		$$
		\lim_{r \rightarrow +\infty} \sup_{i \in \N} \int_{\left\{ x \,\middle\vert \, \rho^p \left( x, x_0 \right) \geq r \right\}} \rho^p \left(x, x_0 \right) \mu_i (dx)  = 0.
		$$
		
		\item For each continuous function $f: M \rightarrow \R$ satisfies that there exist $x_0 \in M$ and $c>0$ such that 
		$$
		\left\vert f(x) \right\vert \leq c \left( 1+ \rho^p \left(x,x_0 \right) \right), \quad \forall x \in M,
		$$
		we have 
		$$
		\lim_{i \rightarrow \infty} \int_M f(x) \mu_i(dx) = \int_M f(x) \mu(dx).
		$$
	\end{enumerate}
\end{proposition}

%\medskip

The next theorem reveals the relation between the narrowly convergence and the almost surely convergence. See, for instance, \cite[Theorem 6.7]{billingsley1999convergence}.
\begin{theorem}[Skorokhod's Representation Theorem] \label{Skorokhod Representation Theorem}
	Let $\left\{ \mu_i \right\}_{i =1}^{\infty} \subset \PP(M)$ be a sequence of probability measures such that $\mu_i$ converges to some $\mu \in \PP(M)$ narrowly. There exist $M$-valued random variables $\left\{ X_i \right\}_{i=1}^{\infty}$ and $X$, whose distributions are $\left\{ \mu_i \right\}_{i =1}^{\infty}$ and $\mu$ respectively. Then, $X_i$ converges to $X$ $\mu$-almost surely.
\end{theorem}

Based on Theorem \ref{Skorokhod Representation Theorem}, we can easily derive Corollary \ref{lemma of s representation theorem}.
\begin{corollary}\label{lemma of s representation theorem}
	Let $\left\{ \mu_i \right\}_{i =1}^{\infty} \subset \PP(M)$ be a sequence of probability measures such that $\mu_i$ converges to some $\mu \in \PP(M)$ narrowly. For any lower semi-continuous and bounded function $g: M \rightarrow \R$, we have
	$$
	\liminf_{i \rightarrow \infty} \int_M g(x) \mu_i(dx) \geq \int_M g(x) \mu(dx).
	$$
\end{corollary}
%\begin{proof}
%By Theorem \ref{Skorokhod Representation Theorem}, there exist random variables $\left\{ X_i \right\}_{i=1}^{\infty}$ and $X$, which are distributions of $\left\{ \mu_i \right\}_{i =1}^{\infty}$ and $\mu$ respectively, such that $X_i$ converges to $X$ $\mu$-almost surely.
%Since $g$ is lower semi-continuous, we have $g(X) \leq \liminf_{i \rightarrow \infty} g \left( X_i \right)$.
%By Fatou's Lemma, we obtain that
%\begin{align*}
%	\int_M g(x) \mu(dx)
%	& = E \left[ g(X) \right] \leq E \left[ \liminf_{i \rightarrow \infty} g \left( X_i \right)\right] \\
%	& \leq \liminf_{i \rightarrow \infty} E \left[ g\left( X_i \right) \right] = \liminf_{i \rightarrow \infty} \int_M g(x) \mu_i (dx).
%\end{align*}
%\end{proof}
%\medskip
\medskip
Finally, we recall the disintegration theorem. See, for instance, \cite[Theorem 8.5]{2013}.
\begin{theorem} [Disintegration Theorem] \label{disintegration_thm}
	Let $X$ and $Y$ be Radon separable metric spaces. Let $\mu$ be a Borel probability measure on $X$. Let $\pi: X \rightarrow Y$ be a Borel map. Define $\nu=\pi \sharp \mu \in \mathcal{P}(Y)$. Then, there exists a $\mu$-almost everywhere uniquely determined family of Borel probability measures $\left\{\nu_y\right\}_{y \in Y} \subset \mathcal{P}(X)$ such that
	$$
	\begin{array}{cc}
		\nu_y\left(X \backslash \pi^{-1}(y)\right)=0, \quad \text { for } \mu-\text { a.e. } y \in Y, \\
		\quad \int_X g(x) \mu(d x) = \int_Y \left( \int_{\pi^{-1}(y)} g(x) \nu_y(d x)\right) \nu(d y)
	\end{array}
	$$
	for every Borel map $g: X \rightarrow[0,+\infty]$.
\end{theorem}

\section{Completion of the Proof of Proposition \ref{3.7_1_1}} \label{Completion of the Proof of Proposition 3.7_1_1}
Define the function $g: \Gamma_T^R \rightarrow \bar{\R}$ by
$$
\gamma \mapsto \int_0^T \int_{\R^n} \left\vert u \right\vert^q \hat{q}_{t,\gamma(t)} \left(d u \right) dt,
$$
which is a positive measurable function. For any Borel subset $V \subset \Gamma^R_T \times \PP^R_U$, there always exists a measure $\mu_{V} \in \PP^R_U$ such that
\begin{align*}
& \int_{\Gamma^R_T \times \PP^R_U} \int_0^T \int_{\R^n} \mathbf{1}_{V} (\gamma, \mu) \left\vert u \right\vert^q \mu (dt,du) P \left( d \gamma, d \mu \right) \\
\leq & P(V) \int_0^T \int_{\R^n} \left\vert u \right\vert^q \mu_V (dt,du) \leq P(V) R.
\end{align*}
Moreover, we have
\begin{align*}
& \int_{\Gamma^R_T } \mathbf{1}_{\pi_1(V)} (\gamma) \, g \left( \gamma \right) \pi_1 \sharp P \left( d \gamma \right) \\	
= & \int_{\Gamma^R_T } \int_0^T \int_{\R^n} \mathbf{1}_{\pi_1(V)} (\gamma) \left\vert u \right\vert^q \hat{q}_{t,\gamma(t)} \left(d u \right)\, dt \,\pi_1 \sharp P \left( d \gamma \right) \\
%= &  \int_0^T \int_{\T^d} \int_{\R^n} \mathbf{1}_{\pi_1(V)} \left( e_t^{-1}(x) \right) \left\vert u \right\vert^q \hat{q}_{t,x} (du) \, e_t \sharp \left( \pi_1 \sharp P \right) (dx)\, dt\\
= & T \int_0^T \int_{\T^d} \int_{\R^n} \mathbf{1}_{\pi_1(V)} \left( e_t^{-1}(x) \right) \left\vert u \right\vert^q \hat{q}_{t,x} (du) \, \eta_{1,2} (dt,dx)\\
= & T \int_0^T \int_{\T^d} \int_{\R^n} \mathbf{1}_{\pi_1(V)} \left( e_t^{-1}(x) \right) \left\vert u \right\vert^q \eta(dt,dx,du)\\
= & \int_{\Gamma^R_T \times \PP^R_U} \int_0^T \int_{\R^n} \mathbf{1}_{ \pi_1(V) } (\gamma) \left\vert u \right\vert^q \mu (dt,du) \,P \left( d \gamma, d \mu \right) \leq R \, P(V),
\end{align*}
indicating that $g$ is integrable.
By Vitali-Carath\'eodory theorem, for any $\varepsilon >0$, there exists a lower semi-continuous function $\hat{g}: \Gamma_T^R \rightarrow \R$ with $g(\gamma) \leq \hat{g} (\gamma)$ for $\pi_1 \sharp P$-$a.e. \,\, \gamma \in \Gamma_T^R$ such that
$$
\int_{\Gamma_T^R} \hat{g}(\gamma) - g(\gamma) \,\, \pi_1 \sharp P(d \gamma) \leq \varepsilon.
$$
For $\pi_1 \sharp P$-$a.e. \,\, \gamma_0 \in \supp \left(\pi_1 \sharp P \right)$, there exists a unique measure $\mu^0 \in \PP^R_U$ such that the pair $\left( \gamma_0, \mu^0 \right) \in \supp(P)$ by Proposition \ref{properties_PR_1}(4).
Denote by $V$ the open neighborhood of $\left( \gamma_0, \mu^0 \right)$, we have $P(V) > 0$ and
\begin{align*}
g(\gamma_0) 
& \leq \hat{g} (\gamma_0) \leq \liminf_{V \rightarrow \left( \gamma_0, \mu^0 \right)} \frac{1}{P(V)}\int_{V} \hat{g}(\gamma) P \left( d\gamma, d\mu \right) \\
& \leq \liminf_{V \rightarrow \left( \gamma_0, \mu^0 \right)} \frac{1}{P(V)}\int_{V} g(\gamma) P \left( d\gamma, d\mu \right) + \varepsilon \leq R+ \varepsilon.
\end{align*}
Let $\varepsilon$ tend to 0. We conclude that $\LL_{[0,T]} \otimes \hat{q}_{t, \gamma(t)} \in \PP^R_U$ for $\pi_1 \sharp P$-$a.e. \,\, \gamma \in \supp \left(\pi_1 \sharp P \right)$. Let $\LL_{[0,T]} \otimes \hat{q}_{t, \gamma(t)}$ equal any measure that belongs to $\PP^R_U$ when $\gamma$ belongs to the null set.
We can conclude that $\hat{P} \in \PP \left( \Gamma^R_T \times \PP^R_U \right)$ since the difference in the null set has no influence on the following analysis by (F\ref{f1_1}) and (L\ref{L3_1}).

It is clear that $e_0 \sharp \left( \pi_1 \sharp \hat{P} \right) = e_0 \sharp \left( \pi_1 \sharp P \right) = m_0$.
Moreover, for any $t \in [0,T]$ and any open set $N \subset \Gamma^R_T \times \PP^R_U$, we can deduce that
\begin{equation*}
\begin{aligned} 
	& \int_{\Gamma^R_T } \int_0^t \int_{\R^n} \mathbf{1}_{\pi_1(N)} (\gamma) f \left( \gamma(s), u ,m(s)\right) \hat{q}_{s,\gamma(s)} \left(d u \right)\, ds \,\pi_1 \sharp P \left( d \gamma \right) \\
	= &  \int_0^t \int_{\T^d} \int_{\R^n} \mathbf{1}_{\pi_1(N)} \left( e_s^{-1}(x) \right) f \left( x, u ,m(s)\right) \hat{q}_{s,x} (du)\,e_s \sharp \left( \pi_1 \sharp P \right) (dx)\, ds\\
%	= & T \int_0^T \int_{\T^d} \int_{\R^n} \mathbf{1}_{\pi_1(N)} \left( e_s^{-1}(x) \right) \mathbf{1}_{[0,t]} (s) f \left( x, u,m(s) \right) \hat{q}_{s,x} (du) \, \eta_{1,2} (ds,dx)\\
	= & T \int_0^T \int_{\T^d} \int_{\R^n} \mathbf{1}_{\pi_1(N)} \left( e_s^{-1}(x) \right) \mathbf{1}_{[0,t]}(s) f \left( x, u,m(s) \right) \eta(ds,dx,du)\\
%	= & \int_{\Gamma^R_T \times \PP^R_U} \int_0^T \int_{\R^n} \mathbf{1}_{ \pi_1(N) } (\gamma) \mathbf{1}_{[0,t]}(s) f \left( \gamma(s), u,m(s) \right) \mu (ds,du) \,P \left( d \gamma, d \mu \right) \\
	= & \int_{\pi_1(N) \times \PP^R_U} \int_0^t \int_{\R^n} f \left( \gamma(s), u ,m(s)\right) \mu (ds,du) \,P \left( d \gamma, d \mu \right).
\end{aligned}
\end{equation*}
Based on this equality, for any $v \in \R^d$, we have
\begin{align*}
	& \int_{N} \left\langle v, \gamma(t) - \gamma(0) - \int_0^t \int_{\R^n} f \left( \gamma(s), u ,m(s) \right) \mu \left( ds,du \right) \right\rangle \hat{P} \left( d\gamma, d \mu \right)   \\
	= & \int_{\Gamma_T^R \times \PP_{U}^{R}} \mathbf{1}_{N} \left( \gamma, \mu \right) \left\langle v, \gamma(t) - \gamma(0) - \int_0^t \int_{\R^n} f \left( \gamma(s), u ,m(s) \right) \mu \left( ds,du \right) \right\rangle \hat{P} \left( d\gamma, d \mu \right)   \\
	= & \int_{\Gamma^R_T } \mathbf{1}_{\pi_1(N)} \left( \gamma \right) \left \langle v, \gamma(t) - \gamma(0) - \int_0^t \int_{\R^n} f \left( \gamma(s), u,m(s) \right) \hat{q}_{s, \gamma(s)} \left( du \right) ds \right\rangle \pi_1 \sharp P \left( d\gamma \right) \\
	= &  \int_{\pi_1(N) \times \PP^R_U} \left \langle v, \gamma(t) - \gamma(0) - \int_0^t \int_{\R^n} f \left( \gamma(s), u ,m(s) \right) \mu \left( ds,du \right) \right \rangle P \left( d\gamma, d \mu \right) =0.
\end{align*}
Thus, we obtain that $\hat{P} \in \PP_R(m)$.

\medskip
\medskip

%%%%%%%%%%%%%%%%%%%%%%%%%%%%%%%%%%%%
\noindent {\bf Statements and Declarations.} 
This paper is the authors' original work and has not been
published or submitted simultaneously elsewhere.

\medskip

\noindent {\bf Competing Interests.}
Authors have no financial interests that are directly related to the work submitted for publication. We have no conflicts of interest to disclose.

\medskip

\noindent {\bf The Data Availability Statement.}
No datasets were generated or analysed during the current study.

\medskip

\noindent {\bf Acknowledgements.} 
Cristian Mendico was partially supported by Istituto Nazionale di Alta Matematica, INdAM-GNAMPA project 2023/2024, by the MIUR Excellence Department Project MatMod@TOV awarded to the Department of Mathematics, University of Rome Tor Vergata, CUP E83C23000330006, and by the King Abdullah University of Science and Technology (KAUST) project CRG2021-4674 "Mean-Field Games: models, theory and computational aspects". Kaizhi Wang is supported by NSFC Grant No. 12171315, 11931016, and Natural Science Foundation of Shanghai No. 22ZR1433100.

\medskip
%%%%%%%%%%%%%%%%%%%%%%%%%%%%%%%%%%%%
\bibliographystyle{plain}
\bibliography{references}

\end{document}